\documentclass[10pt]{article}
\usepackage{color}
\usepackage{amsmath}
\usepackage{graphicx}
\usepackage{footnote}
\usepackage{subscript}
\usepackage{amsfonts,color}
\usepackage{amsmath,amssymb}
\usepackage{epsfig}
\usepackage{amssymb}
\usepackage{mathtools}
\usepackage{amsthm,wasysym}
\usepackage[english]{babel}
\makeatletter
\usepackage{float}
\usepackage{wrapfig}
\restylefloat{figure}

\makeatletter
\let\@fnsymbol\@arabic
\makeatother

\usepackage{bbm}
\newcommand{\id}{{\boldsymbol{\mathbbm{1}}}}

\newcommand{\tr}{{\rm tr}}
\newcommand{\dev}{{\rm dev}}
\newcommand{\sym}{{\rm sym}}
\newcommand{\skw}{{\rm skew}}

\newcommand{\Curl}{{\rm Curl}}

\def\dd{\displaystyle}

\newtheorem{theorem}{Theorem}[section]

\newtheorem{remark}[theorem]{Remark}
\newtheorem{proposition}[theorem]{Proposition}
\newtheorem{corollary}[theorem]{Corollary}

\def\barr{\begin{array}}

	\def\earr{\end{array}}
\def\bec#1{\begin{equation}\label{#1}}
\def\becn{\begin{equation*}}
\def\endec{\end{equation}}
\def\endecn{\end{equation*}}
\def\dd{\displaystyle}
\def\bfm#1{\mbox{\boldmat}}

\begin{document}

	\title{The isotropic  Cosserat  shell model  including terms up to  $O(h^5)$. \\
		Part II: Existence of minimizers}
\author{  Ionel-Dumitrel Ghiba\thanks{Corresponding author: Ionel-Dumitrel Ghiba,  \ Department of Mathematics,  Alexandru Ioan Cuza University of Ia\c si,  Blvd.
		Carol I, no. 11, 700506 Ia\c si,
		Romania; and  Octav Mayer Institute of Mathematics of the
		Romanian Academy, Ia\c si Branch,  700505 Ia\c si, email:  dumitrel.ghiba@uaic.ro} \quad and \quad Mircea B\^irsan\thanks{Mircea B\^irsan, \ \  Lehrstuhl f\"{u}r Nichtlineare Analysis und Modellierung, Fakult\"{a}t f\"{u}r Mathematik,
		Universit\"{a}t Duisburg-Essen, Thea-Leymann Str. 9, 45127 Essen, Germany; and Department of Mathematics,  Alexandru Ioan Cuza University of Ia\c si,  Blvd.
		Carol I, no. 11, 700506 Ia\c si,
		Romania;  email: mircea.birsan@uni-due.de} \quad
	and  \quad     Peter Lewintan\,\thanks{Peter Lewintan,  \ \  Lehrstuhl f\"{u}r Nichtlineare Analysis und Modellierung, Fakult\"{a}t f\"{u}r
	Mathematik, Universit\"{a}t Duisburg-Essen,  Thea-Leymann Str. 9, 45127 Essen, Germany, email: peter.lewintan@uni-due.de} \\
	and  \quad      Patrizio Neff\,\thanks{Patrizio Neff,  \ \ Head of Lehrstuhl f\"{u}r Nichtlineare Analysis und Modellierung, Fakult\"{a}t f\"{u}r
		Mathematik, Universit\"{a}t Duisburg-Essen,  Thea-Leymann Str. 9, 45127 Essen, Germany, email: patrizio.neff@uni-due.de}
}
\maketitle
\begin{abstract}
 We show the existence of global minimizers for a geometrically nonlinear  isotropic elastic Cosserat 6-parameter shell model. The proof of the main theorem is based on the direct methods of the calculus of variations using essentially the convexity of the energy in the nonlinear strain and curvature measures. We first show the existence of the solution for the theory including $O(h^5)$ terms. The energy allows us to  show the coercivity for terms up to order $O(h^5)$ and the convexity of the energy. Secondly,  we  consider only that part of the energy  including $O(h^3)$ terms. In this case the obtained minimization problem is not the same as those previously considered in the literature, since  the influence of the curved initial shell configuration appears explicitly  in the expression of the coefficients  of the energies for the reduced two-dimensional variational problem and additional mixed bending-curvature and curvature terms are present. While in the theory including $O(h^5)$ the conditions on the thickness $h$ are those considered in the  modelling process and they are independent of the constitutive parameter,   in the $O(h^3)$-case the coercivity is proven on some more restrictive conditions under the thickness $h$.

 \medskip

 \noindent\textbf{Keywords:}
 geometrically nonlinear Cosserat shell, 6-parameter resultant shell, in-plane drill
 rotations, thin structures, dimensional reduction,  wryness
 tensor, dislocation density tensor, isotropy, calculus of variations, uniform convexity, existence of minimizers
\end{abstract}

\begin{footnotesize}
	\vspace*{-4.4mm}
\tableofcontents
\end{footnotesize}

\section{Introduction}

Shell and plate theories are intended for the study of thin bodies, i.e.,  bodies in which the thickness in one direction is much smaller than the dimensions in the other two orthogonal directions. In this follow up paper we investigate the existence of minimizers to a recently developped isotropic Cosserat shell model \cite{birsan2019refined,GhibaNeffPartI}, including higher order terms. The Cosserat shell model naturally includes an independent triad of rigid directors, which are coupled to the shell-deformation. From an engineering point of view,  such models  are   preferred, since the independent rotation field allows for transparent coupling between shell and beam parts. It is interesting that the kinematical structure of 6-parameter shells  \cite{Eremeyev06,Pietraszkiewicz14,birsan2020derivation} (involving the translation vector and rotation tensor) is identical to the kinematical structure of Cosserat shells (defined as material surfaces endowed with a triad of rigid directors describing the orientation of points).
 Using the derivation approach, Neff  { \cite{Neff_plate04_cmt,Neff_plate07_m3as,Neff_membrane_Weinberg07,Neff_membrane_plate03,Neff_membrane_existence03}
 } 
 has modelled and analysed the so-called nonlinear planar-Cosserat shell models, in which a full triad of orthogonal directors, independent of the normal of the shell, is taken into account. The results have been obtained by an 8-parameter ansatz of the deformation through the thickness and consistent analytic integration over the thickness in the case of a flat undeformed shell reference configuration.   In  previous papers, we have extended the modelling from flat shells to the most general case of initially curved shells \cite{birsan2019refined,GhibaNeffPartI}.
Our ansatz allows for a consistent shell model up to order $O(h^5)$ in the shell thickness. Interestingly, all $O(h^5)$-terms in the shell energy depend on the initial curvature of the shell and vanish for a flat shell.  However, all occurring material coefficients of the shell model are uniquely determined in terms of the underlying isotropic three-dimensional Cosserat bulk-model and the given initial geometry of the shell. Thus, we fill a certain gap in the general 6-parameter shell theory, since all hitherto known  models  leave the precise structure of the constitutive equations wide open. In the present paper, we will show that our model is mathematically well-posed in the sense that global minimizers exist.

The topic of existence of solutions for the 2D equations of linear and nonlinear elastic shells has been treated in many works.   The results that can be found in the literature refer to various types of shell models and they employ different techniques, see  {e.g.}, \cite{Koiter60,Koiter69,Simo89.1,Simo92,Steigmann08,Steigmann12,Sansour92,Tiba02,Birsan08,Ibrahim94,Badur-Pietrasz86}.  The existence theory for linear or nonlinear shells is presented in details in the books of Ciarlet \cite{Ciarlet97,Ciarlet98,Ciarlet00}, together with many historical remarks and bibliographic references. A fruitful approach to the existence theory of 2D plate and shell models (obtained as limit cases of 3D models) is the $\Gamma$-convergence analysis of thin structures, see  {e.g.}, \cite{Neff_Chelminski_ifb07,Neff_Danzig05,Neff_Hong_Reissner08,Paroni06}.   By ignoring the Cosserat effects, in order to start with a well-posed three dimensional model, it is  {preferable} to  consider a polyconvex  energy \cite{Ball77} in the three-dimensional formulation of the initial problem.  In this direction, an example is the article \cite{ciarlet2018existence}, see also \cite{ciarlet2013orientation,bunoiu2015existence}, where the Ciarlet-Geymonat energy \cite{ciarlet1982lois} is used. In these articles, no through the thickness integration is performed analytically and no reduced completely two-dimensional minimization problem is presented. The obtained problems are ``two-dimensional" only in the sense that the final problem is to find three vector fields  on a bounded open subset of $\mathbb{R}^2$, but all three-dimensional coordinates remain present in the minimization problem.  By contrast,  when a  nonlinear three-dimensional problem in the Cosserat theory is considered, the three-dimensional problem is well-posed \cite{neff2014existence,Tambaca10,tambavca2010semicontinuity} and permits a complete dimensional reduction.

The classical geometrically nonlinear Kirchhoff-Love model (the Koiter model for short), is  given by the minimization problem  with respect to the midsurface deformation $m:\omega\subset \mathbb{R}^2\to \mathbb{R}^3$ of the type\footnote{For the sake of simplicity, in this overview in the introduction, we have ignored  the dependence of the minimization problems on  $\nabla \Theta$, where $\Theta$ is the diffeomorphism which maps the midsurface $\omega$ of the fictitious Cartesian parameter space onto the midsurface of the curved reference configuration, the focus being to  present only the essential energy terms appearing in the variational problems.}
\begin{equation}\label{Ap7matrix}
\begin{array}{l}
\dd\int_\omega \bigg\{h\bigg(
\mu\rVert    \big({\rm I}_m-{\rm I}_{y_0}\big) \rVert^2  +\dfrac{\,\lambda\,\mu}{\lambda+2\,\mu} \, \mathrm{tr} \Big[ \big({\rm I}_m-{\rm I}_{y_0}\big) \Big]^2\bigg) \vspace{1.2mm}\\\qquad  +\dd\frac{h^3}{12}\bigg(
\mu\rVert    \big({\rm II}_m-{\rm II}_{y_0}\big) \rVert^2  +\dfrac{\,\lambda\,\mu}{\lambda+2\,\mu} \, \mathrm{tr} \Big[ \big({\rm II}_m-{\rm II}_{y_0}\big) \Big]^2\bigg)\bigg\}\, {\rm d}a,
\end{array}
\end{equation}
where ${\rm I}_{m}:=[{\nabla  m}]^T\,{\nabla  m}\in \mathbb{R}^{2\times 2}$ and  ${\rm II}_{m}:\,=\,-[{\nabla  m}]^T\,{\nabla  m}\in \mathbb{R}^{2\times 2}$ are  the matrix representations of the {\it first fundamental form (metric)} and the  {\it  second fundamental form}  {on $m(\omega)$}, respectively. However, this problem  is notoriously ill-posed, since the first membrane term is  non-convex in $\nabla\, m$ and is indeed a non-rank-one elliptic  {energy}. Even the inclusion of the bending terms is not sufficient to regularize the problem \cite{Neff_plate04_cmt,Neff_plate07_m3as}.  The very same problem arises in geometrically nonlinear Reissner-Mindlin (Naghdi) type shell models, which already include an independent director-vector-field that does not coincide with the normal to the surface, as in the Kirchhoff-Love model.

Let us explain the typical situation by looking at representative energy terms for the different models. Assume that $m:\omega\subset \mathbb{R}^2\to \mathbb{R}^3$ is the deformation of the midsurface of a flat shell, $n_m$ is the unit normal to the shell midsurface, the unit vector $d:\omega\subset\mathbb{R}^2\to \mathbb{R}^3$  is an independent director  vector-field, and $\overline{R}:\omega\subset \mathbb{R}^2\to {\rm SO}(3)$ is an independent rotation field. Then the essence of a Kirchhoff-Love planar shell model is represented by the minimization problem  with respect to $m:\omega\subset \mathbb{R}^2\to \mathbb{R}^3$ of the type
\begin{align}\label{KLmatrix}
\hspace*{-3cm}\text{``Kirchhoff-Love type''}\qquad \qquad
\dd\int_\omega \bigg\{h\,&\underbrace{\lVert   (\nabla m\,|\,n_m)^T(\nabla m\,|\,n_m)-\id_3\rVert^2}_{\text{``membrane"}} +\dd\frac{h^3}{12}\underbrace{\rVert    \nabla n_m\rVert^2}_{\text{``bending"}} \bigg\}\, {\rm d}a\\&=\dd\int_\omega \bigg\{h\,\lVert   (\nabla m)^T(\nabla m)-\id_2\rVert^2 +\dd\frac{h^3}{12}\rVert    \nabla n_m\rVert^2\bigg\}\, {\rm d}a.\notag
\end{align}
The essence of the corresponding Reissner-Mindlin problem is represented by the minimization problem  with respect to $(m,d)$ of the type
\begin{equation}\label{RMmatrix}
\hspace*{-3.8cm}\text{``Reissner-Mindlin type''}\qquad \quad\begin{array}{l}
\dd\int_\omega \bigg\{h\,\lVert   (\nabla m\,|\,d)^T(\nabla m\,|\,d)-\id_3\rVert^2  +\dd\frac{h^3}{12}\rVert    \nabla d\rVert^2 \bigg\}\, {\rm d}a.
\end{array}
\end{equation}
And finally, the Cosserat flat shell model has the structure given by the minimization problem  with respect to $(m,\overline{R})$ of the type
\begin{equation}\label{Cmatrix}
\hspace*{-4cm} \text{``Cosserat-shell''}\qquad \qquad \qquad\ \     \begin{array}{l}
\dd\int_\omega \bigg\{h\,\lVert  \overline{R}^T (\nabla m\,|\,\overline{R}\, e_3)-\id_3\rVert^2  +\dd\frac{h^3}{12}\rVert    \nabla \overline{R}\rVert^2 \bigg\}\, {\rm d}a.
\end{array}
\end{equation}
Problems \eqref{KLmatrix} and \eqref{RMmatrix} are non-elliptic with respect to $m$ at given $d$, while problem \eqref{Cmatrix} is even linear with respect to $m$ at given rotation field $\overline{R}$, which is itself controlled by the curvature term  $\rVert    \nabla \overline{R}\rVert^2$. Therefore, in principle, \eqref{Cmatrix} admits  minimizers, while \eqref{KLmatrix} and \eqref{RMmatrix} in general do not.

In view of these mathematical deficiencies, in the literature we find many types of existence theorems, which
  treat certain approximations of \eqref{Ap7matrix}. The above mentioned approach by Ciarlet and his co-authors \cite{ciarlet2018existence,ciarlet2013orientation,bunoiu2015existence} falls into this category. It has already been noted by Neff \cite{Neff_plate04_cmt}, that an independent control of the continuum rotations in quadratic, non-rank-one convex energies like the membrane-term in \eqref{Ap7matrix} is sufficient to resolve the non-rank-one convexity issue. This is precisely, what the Cosserat shell model is incorporating from the outset by considering not a single director as additional independent field, but a triad of rigid directors - the rotation field $\overline{R}\in {\rm SO}(3)$.

  Concerning the geometrically nonlinear theory of elastic Cosserat shells with drilling rotations including $O(h^3)$-terms, there is no existence theorem published in the literature, except \cite{NeffBirsan13}, as far as we are aware of.  Existence results for the related Cosserat model of initially planar shells have been obtained earlier by Neff \cite{Neff_plate04_cmt,Neff_plate07_m3as}.
 For our new model, we search for the minimizing solution pair of class ${\rm H}^1(\omega,\mathbb{R}^3)$ for the translation vector and ${\rm H}^1(\omega,{\rm SO}(3))$ for the rotation tensor. For the proof of existence, we employ the direct methods of the calculus of variations, extensions of the techniques presented in \cite{Neff_plate04_cmt,Neff_plate07_m3as,NeffBirsan13,Birsan-Neff-L54-2014}, coercivity and uniform convexity of the energy in the appropriate geometrically nonlinear strain and curvature measure.  A first task is to show the existence of the solution for the theory including $O(h^5)$-terms. In this case the expression of the energy allows us to have a decent control on each term of the energy density, in order to show the coercivity and the convexity of the energy. A second task is to consider that part of the energy which contains only $O(h^3)$-terms. In this case the obtained minimization problem is not the same as that considered in \cite{Pietraszkiewicz04,Pietraszkiewicz-book04,Pietraszkiewicz10,Eremeyev06,NeffBirsan13,Birsan-Neff-L54-2014}, since additional mixed bending-curvature and curvature energy-terms  are included and the influence of the curved initial shell configuration appears explicitly  in the expression of the coefficients  of the energies for the reduced two-dimensional variational problem. For the $O(h^3)$-model, the problem of coercivity turns out to be more delicate,  since some steps used to prove the coercivity for the $O(h^5)$-model cannot be done in the same manner. As a preparation for the existence proofs we will rewrite the energy in an equivalent form that allows us to  prove the coercivity and convexity of the energy.
  Moreover, for the $O(h^3)$-model,  we need to impose either a stronger assumption on the constitutive parameters or a relation between the thickness and the characteristic length.  This behaviour  highlights the importance and interest for including $O(h^5)$-terms.

\section{The new   geometrically nonlinear   Cosserat shell model}\setcounter{equation}{0}\label{model}

\subsection{Notation}

In this paper,
for $a,b\in\mathbb{R}^n$ we let $\bigl\langle {a},{b} \bigr\rangle _{\mathbb{R}^n}$  denote the scalar product on $\mathbb{R}^n$ with
associated  {(squared)} vector norm $\lVert a\rVert _{\mathbb{R}^n}^2=\bigl\langle {a},{a} \bigr\rangle _{\mathbb{R}^n}$
The standard Euclidean scalar product on  the set of real $n\times  {m}$ second order tensors $\mathbb{R}^{n\times  {m}}$ is given by
$\bigl\langle  {X},{Y} \bigr\rangle _{\mathbb{R}^{n\times  {m}}}={\rm tr}(X\, Y^T)$, and thus the  {(squared)} Frobenius tensor norm is
$\lVert {X}\rVert ^2_{\mathbb{R}^{n\times  {m}}}=\bigl\langle  {X},{X} \bigr\rangle _{\mathbb{R}^{n\times  {m}}}$. In the following we omit the subscripts
$\mathbb{R}^n,\mathbb{R}^{n\times  {m}}$. The identity tensor on $\mathbb{R}^{n \times n}$ will be denoted by $\id_n$, so that
${\rm tr}({X})=\bigl\langle {X},{\id}_n \bigr\rangle $. We let ${\rm Sym}(n)$ and ${\rm Sym}^+(n)$ denote the symmetric and positive definite symmetric tensors, respectively.  We adopt the usual abbreviations of Lie-group theory,  {e.g.},
${\rm GL}(n)=\{X\in\mathbb{R}^{n\times n}\;|\det({X})\neq 0\}$ the general linear group{,} ${\rm SO}(n)=\{X\in {\rm GL}(n)| X^TX=\id_n,\det({X})=1\}$ with
corresponding Lie-algebras $\mathfrak{so}(n)=\{X\in\mathbb{R}^{n\times n}\;|X^T=-X\}$ of skew symmetric tensors
and $\mathfrak{sl}(n)=\{X\in\mathbb{R}^{n\times n}\;| \,\tr({X})=0\}$ of traceless tensors. For all $X\in\mathbb{R}^{n\times n}$ we set ${\rm sym}\, X\,=\frac{1}{2}(X+X^T)\in{\rm Sym}(n)$, $\skw\,X\,=\frac{1}{2}(X-X^T)\in \mathfrak{so}(n)$ and the deviatoric part $\dev \,X\,=X-\frac{1}{n}\;\,\tr(X)\,\id_n\in \mathfrak{sl}(n)$  and  {we have
the orthogonal Cartan-decomposition  of the Lie-algebra $\mathfrak{gl}(n)
$ on the vector space $\mathbb{R}^{n\times n}$}, $
\mathfrak{gl}(n)=\{\mathfrak{sl}(n)\cap {\rm Sym}(n)\}\oplus\mathfrak{so}(n) \oplus\mathbb{R}\!\cdot\! \id_n,$ $
X=\dev\, \sym \,X\,+ \skw\,X\,+\frac{1}{n}\,\tr(X)\, \id_n\,.
$ A matrix having the  three  columns vectors $A_1,A_2, A_3$ will be written as
$
(A_1\,|\, A_2\,|\,A_3).
$ We make use of the operator $\mathrm{axl}: \mathfrak{so}(3)\to\mathbb{R}^3$ associating with a matrix $A\in \mathfrak{so}(3)$ the vector $\mathrm{axl}{\,A}\coloneqq(-A_{23},A_{13},-A_{12})^T$. The inverse of the operator $\mathrm{axl}: \mathfrak{so}(3)\to\mathbb{R}^3$ is denoted by $\mathrm{anti}: \mathbb{R}^3\to\mathfrak{so}(3)$.

 Let $\Omega$ be an open domain of $\mathbb{R}^{3}$. The usual Lebesgue spaces of square integrable functions, vector or tensor fields on $\Omega$ with values in $\mathbb{R}$, $\mathbb{R}^3$, $\mathbb{R}^{3\times 3}$ or ${\rm SO}(3)$, respectively will be denoted by ${\rm L}^2(\Omega;\mathbb{R})$, ${\rm L}^2(\Omega;\mathbb{R}^3)$, ${\rm L}^2(\Omega; \mathbb{R}^{3\times 3})$ and ${\rm L}^2(\Omega; {\rm SO}(3))$, respectively. Moreover, we use the standard Sobolev spaces ${\rm H}^{1}(\Omega; \mathbb{R})$ \cite{Adams75,Raviart79,Leis86}
of functions $u$.  For vector fields $u=\left(    u_1, u_2, u_3\right)^T$ with  $u_i\in {\rm H}^{1}(\Omega)$, $i=1,2,3$,
we define
$
\nabla \,u:=\left(
\nabla\,  u_1\,|\,
\nabla\, u_2\,|\,
\nabla\, u_3
\right)^T.
$
The corresponding Sobolev-space will be denoted by
$
{\rm H}^1(\Omega; \mathbb{R}^{3})$. If a tensor $Q:\Omega\to {\rm SO}(3)$ has the components in ${\rm H}^1(\Omega; \mathbb{R})$, then we mark this by writing  $Q\in{\rm H}^1(\Omega; {\rm SO}(3))$. When writting the norm in the corresponding  Sobolev-space we will specify the space in subscript. The space will be omitted only when the Frobenius norm or scalar product is considered.

\subsection{The deformation of  Cosserat shells}
Let $\Omega_\xi\subset\mathbb{R}^3$ be a three-dimensional {\it shell-like thin domain}. In a fixed  standard base  $e_1, e_2, e_3$   of $
\mathbb{R}^3$, a generic point of $\Omega_\xi$ will be denoted by $(\xi_1,\xi_2,\xi_3)$. The elastic material constituting the shell is assumed to be homogeneous and isotropic and the reference configuration $\Omega_\xi$ is assumed to be a natural state.
The deformation of the body occupying the domain $\Omega_\xi$ is described by a vector map $\varphi_\xi:\Omega_\xi\subset\mathbb{R}^3\rightarrow\mathbb{R}^3$ (\textit{called deformation}) and by a \textit{microrotation}  tensor
$
 \overline{R}_\xi:\Omega_\xi\subset\mathbb{R}^3\rightarrow {\rm SO}(3)\, .
$
We denote the current configuration (deformed configuration) by $\Omega_c:=\varphi_\xi(\Omega_\xi)\subset\mathbb{R}^3$,
see Figure \ref{Fig1}.

\begin{figure}[h!]
	\begin{center}
		\includegraphics[scale=1.3]{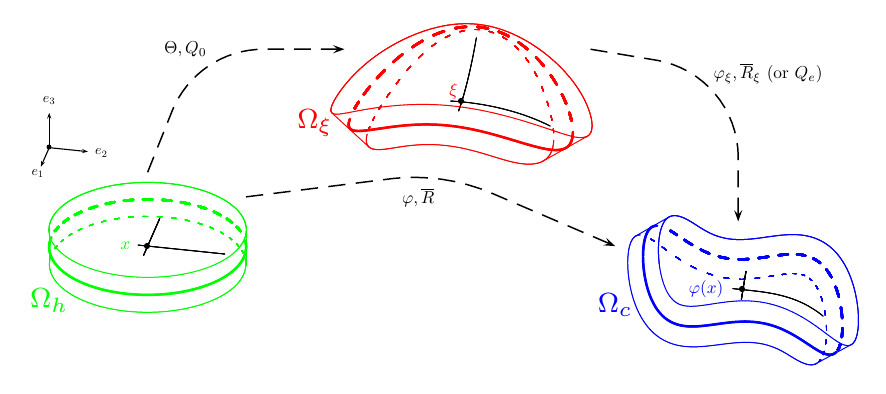}
		% % % 	\includegraphics[scale=0.75]{Fig1.eps}
		% % % 		\put(-330,45){\footnotesize{$O$}} \put(-358,40){\footnotesize{$\boldsymbol{e}_1$}}
		% % % 		\put(-305,59){\footnotesize{${e}_2$}}
		% % % 		\put(-330,85){\footnotesize{${e}_3$}}
		% % % 		\put(-290,35){\footnotesize{${x}$}}
		% % % 		%\put(-350,32){\small{$\omega$}}
		% % % 		\put(-310,5){\footnotesize{$\Omega_h$}}
		% % % 		\put(-197,104){\footnotesize{$\Omega_\xi$}}
		% % % 		\put(-185,150){\small{${\xi}$}}
		% % % 		\put(-60,14){\footnotesize{$\Omega_c$}}
		% % % 		\put(-75,48){\footnotesize{$\varphi({x})\!=\!\varphi_\xi({\xi})$}}
		% % % 		\put(-295,90){\footnotesize{$\Theta  ,  Q_0$}}
		% % % 		\put(-190,70){\footnotesize{$\varphi  , \overline{ R}$}}
		% % % 		\put(-75,120){\footnotesize{$\varphi_\xi  ,  \overline{R}_\xi\ \ (\text{or }\  Q_e)$}}
		\put(-325,50){\textcolor{green}{$\omega$}}
		\caption{\footnotesize The shell in its   initial configuration $\Omega_\xi$, the shell in the deformed configuration $\Omega_c$, and the fictitious planar Cartesian reference configuration $\Omega_h$. Here,  $\overline{R}_{\xi}  $ is the elastic rotation field, ${Q}_{0}$ is the  initial rotation from the fictitious planar Cartesian reference configuration to the initial  configuration $\Omega_\xi$, and $\overline{R}$ is the total rotation field from the fictitious planar Cartesian reference configuration to the deformed configuration $\Omega_c$.}
		\label{Fig1}       % Give a unique label
	\end{center}
\end{figure}

In what follows, we consider  the \textit{fictitious Cartesian (planar) configuration} $\Omega_h$ of the body. This parameter domain $\Omega_h\subset\mathbb{R}^3$ is a right cylinder of the form
$$\Omega_h=\left\{ (x_1,x_2,x_3) \,\Big|\,\, (x_1,x_2)\in\omega, \,\,\,-\dfrac{h}{2}\,< x_3<\, \dfrac{h}{2}\, \right\} =\,\,\dd\omega\,\times\left(-\frac{h}{2},\,\frac{h}{2}\right),$$
where  $\omega\subset\mathbb{R}^2$ is a bounded domain with Lipschitz boundary
$\partial \omega$ and the constant length $h>0$ is the \textit{thickness of the shell}.
For shell--like bodies we consider   the  domain $\Omega_h $ to be {thin}, i.e., the thickness $h$ is {small}.
We assume furthermore that there exists  a $C^1$-diffeomorphism $\Theta:\mathbb{R}^3\rightarrow\mathbb{R}^3$  in the specific form
\begin{equation}\label{defTheta}
\Theta(x_1,x_2,x_3)\,=\,y_0(x_1,x_2)+x_3\ n_0(x_1,x_2), \ \ \ \ \ n_0\,=\,\dd\frac{\partial_{x_1}y_0\times \partial_{x_2}y_0}{\lVert \partial_{x_1}y_0\times \partial_{x_2}y_0\rVert }\, ,
\end{equation}
where $y_0:\omega\to \mathbb{R}^3$ is a function of class $C^2(\omega)$,  {so that $\Theta$}  maps the fictitious planar Cartesian parameter space $\Omega_h$ onto  the initially curved reference configuration of the shell  $\Theta(\Omega_h)=\Omega_\xi$,   $\Theta(x_1,x_2,x_3)=(\xi_1,\xi_2,\xi_3)$.
The diffeomorphism
$\Theta$ maps the midsurface $\omega$ of the fictitious Cartesian  parameter space $\Omega_h$ onto the midsurface $\omega_\xi=y_0(\omega)$ of $\Omega_\xi$ and $n_0$ is the unit normal vector to $\omega_\xi$.   For simplicity and  where no confusions may arise,  we will omit subsequently  to write explicitly   the arguments $(x_1,x_2, x_3)$ of the diffeomorphism  $\Theta$ or we will specify only its dependence  on $x_3$. We use the polar decomposition \cite{neff2013grioli} of $\nabla_x \Theta(x_3)$  and write
$
\nabla_x \Theta(x_3)={Q}_0(x_3)\, U_0(x_3)\, ,$
${Q}_0(x_3)={\rm polar}{(\nabla_x \Theta)(x_3)}\in {\rm SO}(3 ),$   $U _0(x_3)\in \rm{Sym}^+(3).
$ Let us remark that
\begin{align}
\nabla_x \Theta(x_3)&\,=\,(\nabla y_0|n_0)+x_3(\nabla n_0|0) \, \  \forall\, x_3\in \left(-\frac{h}{2},\frac{h}{2}\right),
\ \
\nabla_x \Theta(0)\,=\,(\nabla y_0|\,n_0),\ \ [\nabla_x \Theta(0)]^{-T}\, e_3\,=n_0,
\end{align}
and that $\det(\nabla y_0|n_0)=\sqrt{\det[ (\nabla y_0)^T\nabla y_0]}$ represents the surface element.

In the following, we consider the  {\it Weingarten map\footnote{We identify the Weingarten map, the first fundamental form and the second fundamental form with their associated matrices in the fixed base vector $e_1,e_2,e_3$.} (or shape operator)}  {on $y_0(\omega)$} defined by  {its associated matrix}
$
{\rm L}_{y_0}\,=\, {\rm I}_{y_0}^{-1} {\rm II}_{y_0}\in \mathbb{R}^{2\times 2},
$
where ${\rm I}_{y_0}:=[{\nabla  y_0}]^T\,{\nabla  y_0}\in \mathbb{R}^{2\times 2}$ and  ${\rm II}_{y_0}:\,=\,-[{\nabla  y_0}]^T\,{\nabla  n_0}\in \mathbb{R}^{2\times 2}$ are  the matrix representations of the {\it first fundamental form (metric)} and the  {\it  second fundamental form}, respectively.
Then, the {\it Gau{\ss} curvature} ${\rm K}$ of the surface  {$y_0(\omega)$}  is determined by
$
{\rm K} :=\,{\rm det}{({\rm L}_{y_0})}\,
$
and the {\it mean curvature} $\,{\rm H}\,$ through
$
2\,{\rm H}\, :={\rm tr}({{\rm L}_{y_0}}).
$  We also need  the  tensors defined by:
\begin{align}\label{AB}
{\rm A}_{y_0}&:=(\nabla y_0|0)\,\,[\nabla_x \Theta(0)]^{-1}\in\mathbb{R}^{3\times 3}, \qquad
{\rm B}_{y_0}:=-(\nabla n_0|0)\,\,[\nabla_x \Theta(0)]^{-1}\in\mathbb{R}^{3\times 3},
\end{align}
and the so-called \textit{{alternator tensor}} ${\rm C}_{y_0}$ of the surface \cite{Zhilin06}
	\begin{align}
{\rm C}_{y_0}:=\det(\nabla_x \Theta(0))\, [	\nabla_x \Theta(0)]^{-T}\,\begin{footnotesize}\begin{pmatrix}
0&1&0 \\
-1&0&0 \\
0&0&0
\end{pmatrix}\end{footnotesize}\,  [	\nabla_x \Theta(0)]^{-1}.
\end{align}

Now, let us  define the map
$
\varphi:\Omega_h\rightarrow \Omega_c,\  \varphi(x_1,x_2,x_3)=\varphi_\xi( \Theta(x_1,x_2,x_3)).
$
We view $\varphi$ as a function which maps the fictitious  planar reference configuration $\Omega_h$ into the deformed configuration $\Omega_c$.
We also consider the \textit{elastic microrotation}
$
\overline{Q}_{e,s}:\Omega_h\rightarrow{\rm SO}(3),\  \overline{Q}_{e,s}(x_1,x_2,x_3):=\overline{R}_\xi(\Theta(x_1,x_2,x_3))\,.
$

In \cite{GhibaNeffPartI}, by   assuming that
$
\overline{Q}_{e,s}(x_1,x_2,x_3)=\overline{Q}_{e,s}(x_1,x_2)
$
and  considering an \textit{8-parameter quadratic ansatz} in the thickness direction for the reconstructed total deformation $\varphi_s:\Omega_h\subset \mathbb{R}^3\rightarrow \mathbb{R}^3$ of the shell-like body,
 we have obtained a two-dimensional minimization problem in which the energy density is expressed in terms of the  following tensor fields  on the surface $\omega\,$
\begin{align}\label{e55}
\mathcal{E}_{m,s} & :\,=\,    \overline{Q}_{e,s}^T  (\nabla  m|\overline{Q}_{e,s}\nabla_x\Theta(0)\, e_3)[\nabla_x \Theta(0)]^{-1}-\id_3\qquad \qquad \quad\ \,\, \text{({elastic shell strain tensor})} ,  \\
\mathcal{K}_{e,s} & :\,=\,  (\mathrm{axl}(\overline{Q}_{e,s}^T\,\partial_{x_1} \overline{Q}_{e,s})\,|\, \mathrm{axl}(\overline{Q}_{e,s}^T\,\partial_{x_2} \overline{Q}_{e,s})\,|0)[\nabla_x \Theta(0)]^{-1} \qquad \text{(elastic shell bending--curvature tensor)},\notag
\end{align}
where $m:\omega\subset\mathbb{R}^2\to\mathbb{R}^3$ represents the
deformation of the midsurface.  When these measures vanish, the shell undergoes a rigid body motion. Indeed, $\mathcal{K}_{e,s} {=0}$ implies $\partial_{x_1} \overline{Q}_{e,s}=0$, $\partial_{x_2} \overline{Q}_{e,s}=0$, while $\mathcal{E}_{m,s}=0$ leads to $\nabla m=\overline{Q}_{e,s} \nabla y_0$. Since $\overline{Q}_{e,s}$ is constant and $m=\overline{Q}_{e,s}\,y_0+c$, where $c$ is a constant vector field, this means that the shell is in a rigid body motion with constants translation $c$ and constant rotation $\overline{Q}_{e,s}$.

\subsection{Formulation of the minimization problem}

In   \cite{GhibaNeffPartI}, we have obtained the following two-dimensional minimization problem   for the
deformation of the midsurface $m:\omega
\,{\to}\,
\mathbb{R}^3$ and the microrotation of the shell
$\overline{Q}_{e,s}:\omega
\,{\to}\,
\textrm{SO}(3)$ solving on $\omega
\,\subset\mathbb{R}^2
$: minimize with respect to $ (m,\overline{Q}_{e,s}) $ the  functional
\begin{equation}\label{e89}
I(m,\overline{Q}_{e,s})\!=\!\! \int_{\omega}   \!\!\Big[  \,
W_{\mathrm{memb}}\big(  \mathcal{E}_{m,s}  \big) +W_{\mathrm{memb,bend}}\big(  \mathcal{E}_{m,s} ,\,  \mathcal{K}_{e,s} \big)   +
W_{\mathrm{bend,curv}}\big(  \mathcal{K}_{e,s}    \big)
\Big] \,{\rm det}(\nabla y_0|n_0)       \,  {{\rm d}a} - \overline{\Pi}(m,\overline{Q}_{e,s})\,,
\end{equation}
where the  membrane part $\,W_{\mathrm{memb}}\big(  \mathcal{E}_{m,s} \big) \,$, the membrane--bending part $\,W_{\mathrm{memb,bend}}\big(  \mathcal{E}_{m,s} ,\,  \mathcal{K}_{e,s} \big) \,$ and the bending--curvature part $\,W_{\mathrm{bend,curv}}\big(  \mathcal{K}_{e,s}    \big)\,$ of the shell energy density are given by
\begin{align}\label{e90}
W_{\mathrm{memb}}\big(  \mathcal{E}_{m,s} \big)=& \, \Big(h+{\rm K}\,\dfrac{h^3}{12}\Big)\,
W_{\mathrm{shell}}\big(    \mathcal{E}_{m,s} \big),\vspace{2.5mm}\notag\\
W_{\mathrm{memb,bend}}\big(  \mathcal{E}_{m,s} ,\,  \mathcal{K}_{e,s} \big)=& \,   \Big(\dfrac{h^3}{12}\,-{\rm K}\,\dfrac{h^5}{80}\Big)\,
W_{\mathrm{shell}}  \big(   \mathcal{E}_{m,s} \, {\rm B}_{y_0} +   {\rm C}_{y_0} \mathcal{K}_{e,s} \big)  \\&
-\dfrac{h^3}{3} \mathrm{ H}\,\mathcal{W}_{\mathrm{shell}}  \big(  \mathcal{E}_{m,s} ,
\mathcal{E}_{m,s}{\rm B}_{y_0}+{\rm C}_{y_0}\, \mathcal{K}_{e,s} \big)+
\dfrac{h^3}{6}\, \mathcal{W}_{\mathrm{shell}}  \big(  \mathcal{E}_{m,s} ,
( \mathcal{E}_{m,s}{\rm B}_{y_0}+{\rm C}_{y_0}\, \mathcal{K}_{e,s}){\rm B}_{y_0} \big)\vspace{2.5mm}\notag\\&+ \,\dfrac{h^5}{80}\,\,
W_{\mathrm{mp}} \big((  \mathcal{E}_{m,s} \, {\rm B}_{y_0} +  {\rm C}_{y_0} \mathcal{K}_{e,s} )   {\rm B}_{y_0} \,\big),  \vspace{2.5mm}\notag\\
W_{\mathrm{bend,curv}}\big(  \mathcal{K}_{e,s}    \big) =&  \,\Big(h-{\rm K}\,\dfrac{h^3}{12}\Big)\,
W_{\mathrm{curv}}\big(  \mathcal{K}_{e,s} \big)    +  \Big(\dfrac{h^3}{12}\,-{\rm K}\,\dfrac{h^5}{80}\Big)\,
W_{\mathrm{curv}}\big(  \mathcal{K}_{e,s}   {\rm B}_{y_0} \,  \big)  + \,\dfrac{h^5}{80}\,\,
W_{\mathrm{curv}}\big(  \mathcal{K}_{e,s}   {\rm B}_{y_0}^2  \big),\notag
\end{align}
with
\begin{align}\label{quadraticforms}
W_{\mathrm{shell}}(  S) & =   \mu\,\lVert  \mathrm{sym}\,S\rVert ^2 +  \mu_{\rm c}\lVert \mathrm{skew}\,S\rVert ^2 +\dfrac{\lambda\,\mu}{\lambda+2\,\mu}\,\big[ \mathrm{tr}   (S)\big]^2,\\
\mathcal{W}_{\mathrm{shell}}(  S,  T)& =   \mu\,\bigl\langle  \mathrm{sym}\,S,\,\mathrm{sym}\,   \,T \bigr\rangle   +  \mu_{\rm c}\bigl\langle \mathrm{skew}\,S,\,\mathrm{skew}\,   \,T \bigr\rangle   +\,\dfrac{\lambda\,\mu}{\lambda+2\,\mu}\,\mathrm{tr}   (S)\,\mathrm{tr}   (T),  \notag\vspace{2.5mm}\\
W_{\mathrm{mp}}(  S)&= \mu\,\lVert \mathrm{sym}\,S\rVert ^2+  \mu_{\rm c}\lVert \mathrm{skew}\,S\rVert ^2 +\,\dfrac{\lambda}{2}\,\big[  \tr(S)\,\big]^2=
\mathcal{W}_{\mathrm{shell}}(  S)+ \,\dfrac{\lambda^2}{2\,(\lambda+2\,\mu)}\,[\mathrm{tr} (S)]^2,\notag\vspace{2.5mm}\\
W_{\mathrm{curv}}(  S )&=\mu\, L_{\rm c}^2 \left( b_1\,\lVert  \dev\,\text{sym} \,S\rVert ^2+b_2\,\lVert \text{skew}\,S\rVert ^2+b_3\,
[\tr (S)]^2\right) .\notag
\end{align}

The parameters $\mu$ and $\lambda$ are the \textit{Lam\'e constants}
of classical isotropic elasticity, $\kappa=\frac{2\,\mu+3\,\lambda}{3}$ is the \textit{infinitesimal bulk modulus}, $b_1, b_2, b_3$ are \textit{non-dimensional constitutive curvature coefficients (weights)}, $\mu_{\rm c}\geq 0$ is called the \textit{{Cosserat couple modulus}} and ${L}_{\rm c}>0$ introduces an \textit{{internal length} } which is {characteristic} for the material,  {e.g.}, related to the grain size in a polycrystal. The
internal length ${L}_{\rm c}>0$ is responsible for \textit{size effects} in the sense that smaller samples are relatively stiffer than
larger samples. If not stated otherwise, we assume that $\mu>0$, $\kappa>0$, $\mu_{\rm c}>0$, $b_1>0$, $b_2>0$, $b_3> 0$. Here, $\mu\,L_{\rm c}^2$  plays the role of a {\it dimensional agreement factor}. Without loss of  generality, we may assume that $0<b_1<1$, $0<b_2<1$, $0<b_3<1$. All  constitutive coefficients  are deduced from the three-dimensional formulation, without using any a posteriori fitting of some two-dimensional constitutive coefficients.

The potential of applied external loads $ \overline{\Pi}(m,\overline{Q}_{e,s}) $ appearing in \eqref{e89} is expressed by
\begin{align}\label{e4o}
\overline{\Pi}(m,\overline{Q}_{e,s})\,=&\, \, \Pi_\omega(m,\overline{Q}_{e,s}) + \Pi_{\gamma_t}(m,\overline{Q}_{e,s})\,,\qquad\textrm{with}   \\
\Pi_\omega(m,\overline{Q}_{e,s}) \,=&\, \dd\int_{\omega}\bigl\langle  {f}, u \bigr\rangle \,  {{\rm d}a} + \Lambda_\omega(\overline{Q}_{e,s})\qquad \text{and}\qquad
\Pi_{\gamma_t}(m,\overline{Q}_{e,s})\,=\, \dd\int_{\gamma_t}\bigl\langle  {t},  u \bigr\rangle \, ds + \Lambda_{\gamma_t}(\overline{Q}_{e,s})\,,\notag
\end{align}
where $ u(x_1,x_2) \,=\, m(x_1,x_2)-y_0(x_1,x_2) $ is the displacement vector of the midsurface,   {$\Pi_\omega(m,\overline{Q}_{e,s})$ is the total potential of the external surface loads $f$ and of the external applied body couples $\Lambda_\omega(\overline{Q}_{e,s})$, while  $\Pi_{\gamma_t}(m,\overline{Q}_{e,s})$ is the total potential of the external boundary loads $t$ and of the external boundary couples $ \Lambda_{\gamma_t}(\overline{Q}_{e,s})$}.  Here, $ \gamma_t $ and $ \gamma_d $ are nonempty subsets of the boundary of $ \omega $ such that $   \gamma_t \cup \gamma_d= \partial\omega $ and $ \gamma_t \cap \gamma_d= \emptyset $\,. On $ \gamma_t $ we have considered traction boundary conditions, while on $ \gamma_d $ we have the Dirichlet-type boundary conditions: \begin{align}\label{boundary}
m\big|_{\gamma_d}&=m^*,\ \  \ \ \ \ \ \ \text{simply supported (fixed, welded)}, \qquad \qquad
 \overline{Q}_{e,s}\big|_{\gamma_d}=\overline{Q}_{e,s}^*,\ \  \ \ \ \ \ \ \text{clamped\footnotemark},
\end{align}
\footnotetext{The existence theory works also for free microrotations at the boundary since ${\rm SO}(3)$ is a compact manifold.}
where the boundary conditions are to be understood in the sense of traces.

  The functions $\Lambda_\omega\,, \Lambda_{\gamma_t} : {\rm L}^2 (\omega, \textrm{SO}(3))\rightarrow\mathbb{R} $ are expressed in terms of the loads from the three-dimensional parental variational problem, see \cite{GhibaNeffPartI}, and they are assumed to be continuous and bounded operators.

\begin{remark}
Our model \cite{GhibaNeffPartI}   is constructed  under the following assumptions upon the thickness
$$
{h}\,|{\kappa_1}|<\frac{1}{2}\qquad \text{and}\qquad  {h}\,|{\kappa_2}|<\frac{1}{2},
$$
 {where  ${\kappa_1}$ and ${\kappa_2}$ denote  the principal curvatures of the initial underformed surface}.
\end{remark}

We will consider  materials for  which the Poisson ratio  $\nu=\frac{\lambda}{2(\lambda+\mu)}$ and Young's modulus $ E=\frac{\mu(3\,\lambda+2\,\mu)}{\lambda+\mu}$ are such that
$
-\frac{1}{2}<\nu<\frac{1}{2}$ \text{and} $E>0\, .
$
This assumption implies that  $2\,\lambda+\mu>0$.
Under these assumptions on the constitutive coefficients, together with the positivity of $\mu$, $\mu_{\rm c}$, $b_1$, $b_2$ and $b_3$, and the orthogonal Cartan-decomposition  of the Lie-algebra
$
\mathfrak{gl}(3)$, with
\begin{align}\label{e78}
{W}_{\mathrm{shell}}( S)
& \,=\,   \mu\, \lVert  \mathrm{  dev \,sym} \,S\rVert ^2  +  \mu_{\rm c} \lVert  \mathrm{skew}   \,S\rVert ^2 +\,\dfrac{2\,\mu\,(2\,\lambda+\mu)}{3\,(\lambda+2\,\mu)}\,[\mathrm{tr}  (S)]^2,
\end{align}
it follows that there exists the positive constants  $c_1^+, c_2^+, C_1^+$ and $C_2^+$  such that
\begin{align}\label{pozitivdef}
C_1^+ \lVert S\rVert ^2\geq\, W_{\mathrm{shell}}(  S) \geq\, c_1^+ \lVert S\rVert ^2,\qquad \quad
 C_2^+ \lVert S \rVert ^2 \geq\, W_{\mathrm{curv}}(  S )
\geq\,  c_2^+ \lVert S \rVert ^2\qquad \forall\, S\in \mathbb{R}^{3\times 3}.
\end{align}
Hence,
we note
\begin{align}\label{mpmix}
W_{\mathrm{mp}}(  S )&=
{W}_{\mathrm{shell}}( S )+ \,\dfrac{\lambda^2}{2(\lambda+2\,\mu)}\,(\mathrm{tr}  (S ))^2\geq\, {W}_{\mathrm{shell}}(  S )\geq\, c_1^+ \lVert S \rVert ^2.
\end{align}

\section{Existence of minimizers for the Cosserat shell model of order $O(h^5)$}\setcounter{equation}{0}

In order to establish an existence result by the direct methods of the calculus of variations, we need to show the coercivity  of the  elastically stored  shell energy density.

\subsection{Coercivity and  uniform convexity in the theory of order $O(h^5)$}

\begin{proposition}\label{propcoerh5} {\rm [Coercivity in the theory including terms up to order $O(h^5)$]} For sufficiently small values of the thickness $h$ such that  $h|\kappa_1|<\frac{1}{2}$ and $h|\kappa_2|<\frac{1}{2}$
	and for constitutive coefficients  satisfying  $\mu>0, \,\mu_{\rm c}>0$, $2\,\lambda+\mu> 0$, $b_1>0$, $b_2>0$ and $b_3>0$,   the  energy density
	\begin{align}W(\mathcal{E}_{m,s}, \mathcal{K}_{e,s})=W_{\mathrm{memb}}\big(  \mathcal{E}_{m,s} \big)+W_{\mathrm{memb,bend}}\big(  \mathcal{E}_{m,s} ,\,  \mathcal{K}_{e,s} \big)+W_{\mathrm{bend,curv}}\big(  \mathcal{K}_{e,s}    \big)
	\end{align}
	is coercive in the sense that  there exists a constant   $a_1^+>0$  such that
	\begin{equation}\label{26bis}
	W(\mathcal{E}_{m,s}, \mathcal{K}_{e,s})\,\geq\, a_1^+\, \big( \lVert \mathcal{E}_{m,s}\rVert ^2 + \lVert \mathcal{K}_{e,s}\rVert ^2\,\big) ,
	\end{equation}
		where
	$a_1^+$ depends on the constitutive coefficients.
\end{proposition}
\begin{proof}
	In order to prove the coercivity note that the principal curvatures $\kappa_1,\kappa_2$  are the solutions of
	the characteristic equation of
	${\rm L}_{y_0}$, i.e., $\kappa^2-{\rm tr}({{\rm L}_{y_0}})\,\kappa+{\rm det}{({\rm L}_{y_0})}\,=\,\kappa^2-2\,{\rm H}\,\kappa+{\rm K}\,=\,0.$
Therefore,	 from the assumptions
$
	h\, |\kappa_1|<\frac{1}{2},$ $ \  h\, |\kappa_2|<\frac{1}{2},
	$
	it follows that
	\begin{align}\label{condition}
	h^2|K|=h^2\, |\kappa_1|\,|\kappa_2|<\frac{1}{4}\qquad \text{and}\qquad 2\,h\, |H|=h\, |\kappa_1+\kappa_2|<1.
	\end{align}

	Therefore,
	$
	h-{\rm K}\,\frac{h^3}{12}>0 \ \ \textrm{and}  \   \ \frac{h^3}{12}\,-{\rm K}\,\frac{h^5}{80}>0
	$ and
		\begin{align}
	W(\mathcal{E}_{m,s}, \mathcal{K}_{e,s})
	\geq\,&\Big(h+{\rm K}\,\dfrac{h^3}{12}\Big)\,
		W_{\mathrm{shell}}\big(    \mathcal{E}_{m,s} \big)+\Big(\dfrac{h^3}{12}\,-{\rm K}\,\dfrac{h^5}{80}\Big)\,
		W_{\mathrm{shell}}  \big(   \mathcal{E}_{m,s} \, {\rm B}_{y_0} +   {\rm C}_{y_0} \mathcal{K}_{e,s} \big)\notag\\&-\dfrac{h^3}{3}\,|\mathrm{H}|\,
		\,|\mathcal{W}_{\mathrm{shell}}  \big(  \mathcal{E}_{m,s} ,
		\mathcal{E}_{m,s}{\rm B}_{y_0}+{\rm C}_{y_0}\, \mathcal{K}_{e,s} \big)|-\dfrac{h^3}{12}\,\,2\,  |\mathcal{W}_{\mathrm{shell}}  \big(  \mathcal{E}_{m,s} ,
		( \mathcal{E}_{m,s}{\rm B}_{y_0}+{\rm C}_{y_0}\, \mathcal{K}_{e,s}){\rm B}_{y_0} \big)|\notag\\&+ \,\dfrac{h^5}{80}\,\,
		W_{\mathrm{mp}} \big((  \mathcal{E}_{m,s} \, {\rm B}_{y_0} +  {\rm C}_{y_0} \mathcal{K}_{e,s} )   {\rm B}_{y_0} \,\big)+\Big(h-{\rm K}\,\dfrac{h^3}{12}\Big)\,
		W_{\mathrm{curv}}\big(  \mathcal{K}_{e,s} \big).
		\end{align}
		Using the Cauchy--Schwarz inequality we deduce
		\begin{align}
		W(\mathcal{E}_{m,s}, \mathcal{K}_{e,s})
			\geq\,&\Big(h+{\rm K}\,\dfrac{h^3}{12}\Big)\,
		W_{\mathrm{shell}}\big(    \mathcal{E}_{m,s} \big)-
		\dfrac{1}{3}\,|\mathrm{H}|\,  \left[h^2\,{W}_{\mathrm{shell}}  \big(  \mathcal{E}_{m,s}\big)\right]^{\frac{1}{2}}\, \left[h^4\,{W}_{\mathrm{shell}}  \big(
		\mathcal{E}_{m,s}{\rm B}_{y_0}+{\rm C}_{y_0}\, \mathcal{K}_{e,s}  \big)\right]^{\frac{1}{2}}\notag\\&+\Big(\dfrac{h^3}{12}\,-{\rm K}\,\dfrac{h^5}{80}\Big)\,
		W_{\mathrm{shell}}  \big(   \mathcal{E}_{m,s} \, {\rm B}_{y_0} +   {\rm C}_{y_0} \mathcal{K}_{e,s} \big)\notag\\&-\dfrac{1}{6}\, \Big[h\,{W}_{\mathrm{shell}}  \big(  \mathcal{E}_{m,s}\big)]^{\frac{1}{2}}
		\, \Big[h^5{W}_{\mathrm{shell}}  \big( ( \mathcal{E}_{m,s}{\rm B}_{y_0}+{\rm C}_{y_0}\, \mathcal{K}_{e,s}){\rm B}_{y_0} \big)]^{\frac{1}{2}}
	\\&+ \,\dfrac{h^5}{80}\,\,
		W_{\mathrm{mp}} \big((  \mathcal{E}_{m,s} \, {\rm B}_{y_0} +  {\rm C}_{y_0} \mathcal{K}_{e,s} )   {\rm B}_{y_0} \,\big)+\Big(h-{\rm K}\,\dfrac{h^3}{12}\Big)\,
		W_{\mathrm{curv}}\big(  \mathcal{K}_{e,s} \big).\notag
	\end{align}
The	arithmetic-geometric mean inequality leads to the estimate
	\begin{align}
W(\mathcal{E}_{m,s}, \mathcal{K}_{e,s})
			\geq\,&\Big(h+{\rm K}\,\dfrac{h^3}{12}-
		\dfrac{h^2}{6}\varepsilon\,|\mathrm{H}|\Big)\,  {W}_{\mathrm{shell}}  \big(  \mathcal{E}_{m,s}\big)+\Big(\dfrac{h^3}{12}\,-{\rm K}\,\dfrac{h^5}{80}-\dfrac{h^4}{6\, \varepsilon}\,\,|\mathrm{H}|\Big)\, {W}_{\mathrm{shell}}  \big(
		\mathcal{E}_{m,s}{\rm B}_{y_0}+{\rm C}_{y_0}\, \mathcal{K}_{e,s}  \big)\notag\\&-\dfrac{h}{12}\delta\,  {W}_{\mathrm{shell}}  \big(  \mathcal{E}_{m,s}\big)-
		\dfrac{h^5}{12\, \delta} {W}_{\mathrm{shell}}  \big( ( \mathcal{E}_{m,s}{\rm B}_{y_0}+{\rm C}_{y_0}\, \mathcal{K}_{e,s}){\rm B}_{y_0} \big)\\&+ \,\dfrac{h^5}{80}\,\,
		W_{\mathrm{mp}} \big((  \mathcal{E}_{m,s} \, {\rm B}_{y_0} +  {\rm C}_{y_0} \mathcal{K}_{e,s} )   {\rm B}_{y_0} \,\big)+\Big(h-{\rm K}\,\dfrac{h^3}{12}\Big)\,
		W_{\mathrm{curv}}\big(  \mathcal{K}_{e,s} \big)\notag
	\quad \forall\, \varepsilon>0\  \text{and} \ \delta>0.
	\end{align}
	Using \eqref{mpmix}, we obtain
	\begin{align}
W(\mathcal{E}_{m,s}, \mathcal{K}_{e,s})
\geq\,&
	\Big(h-\dfrac{h}{12}\delta+{\rm K}\,\dfrac{h^3}{12}-
	\dfrac{h^2}{6}\varepsilon\,|\mathrm{H}|\Big)\,{W}_{\mathrm{shell}}  \big(  \mathcal{E}_{m,s}\big)\notag\\&+\Big(\dfrac{h^3}{12}\,-{\rm K}\,\dfrac{h^5}{80}-\dfrac{h^4}{6\, \varepsilon}\,\,|\mathrm{H}|\Big)\, {W}_{\mathrm{shell}}  \big(
	\mathcal{E}_{m,s}{\rm B}_{y_0}+{\rm C}_{y_0}\, \mathcal{K}_{e,s}  \big)\notag\\&+ \Big(\,\dfrac{h^5}{80}-
	\dfrac{h^5}{12\, \delta}\Big)\,\,
	W_{\mathrm{shell}} \big((  \mathcal{E}_{m,s} \, {\rm B}_{y_0} +  {\rm C}_{y_0} \mathcal{K}_{e,s} )   {\rm B}_{y_0} \,\big)\\&+\Big(h-{\rm K}\,\dfrac{h^3}{12}\Big)\,
	W_{\mathrm{curv}}\big(  \mathcal{K}_{e,s} \big)\qquad \  \forall\, \varepsilon>0\  \text{ and } \ \  \delta>0\notag.
	\end{align}
Taking $\delta=8$ and	 and $\varepsilon=2$ we get\footnote{This step cannot be repeated in the proof of the coercivity up to order $O(h^3)$, since  $W_{\mathrm{shell}} \big((  \mathcal{E}_{m,s} \, {\rm B}_{y_0} +  {\rm C}_{y_0} \mathcal{K}_{e,s} )   {\rm B}_{y_0} \,\big)$ cannot be skipped. This is the reason why we  have to choose another strategy  to obtain the desired estimates.} that
	\begin{align}
	W(\mathcal{E}_{m,s}, \mathcal{K}_{e,s})
	\geq\,&h\,\Big[\dfrac{1}{3}-{\rm K}\,\dfrac{h^2}{12}-
	\dfrac{h}{3}\,|\mathrm{H}|\Big]\,  {W}_{\mathrm{shell}}  \big(  \mathcal{E}_{m,s}\big)+\dfrac{h^3}{12}\Big(1\,-|{\rm K}|\,\dfrac{12\,h^2}{80}-h\,|\mathrm{H}|\Big)\, {W}_{\mathrm{shell}}  \big(
	\mathcal{E}_{m,s}{\rm B}_{y_0}+{\rm C}_{y_0}\, \mathcal{K}_{e,s}  \big)\notag\\&+\Big(h-|{\rm K}|\,\dfrac{h^3}{12}\Big)\,
	W_{\mathrm{curv}}\big(  \mathcal{K}_{e,s} \big).
	\end{align}
	In view of \eqref{condition} and \eqref{pozitivdef}, we deduce
	\begin{align}
W(\mathcal{E}_{m,s}, \mathcal{K}_{e,s})
\geq\,&h\,\dfrac{7}{48}\,  {W}_{\mathrm{shell}}  \big(  \mathcal{E}_{m,s}\big)+\dfrac{h^3}{12}\,\dfrac{37}{80}\, {W}_{\mathrm{shell}}  \big(
	\mathcal{E}_{m,s}{\rm B}_{y_0}+{\rm C}_{y_0}\, \mathcal{K}_{e,s}  \big)+h\,\dfrac{47}{48}\,
	W_{\mathrm{curv}}\big(  \mathcal{K}_{e,s} \big)\notag\\
	\geq \,&h\,\dfrac{7}{48}\, c_1^+\, \lVert   \mathcal{E}_{m,s}\rVert ^2+\dfrac{h^3}{12}\,\dfrac{37}{80}\,c_1^+\, \lVert
	\mathcal{E}_{m,s}{\rm B}_{y_0}+{\rm C}_{y_0}\, \mathcal{K}_{e,s}  \rVert ^2+h\,\dfrac{47}{48}\,  { c_2^+}
\lVert   \mathcal{K}_{e,s} \rVert ^2.
	\end{align}
The desired constant  $a_1^+$ from the conclusion can be chosen as $a_1^+=\min\big\{h\,\dfrac{7}{48}\, c_1^+,h\,\dfrac{47}{48}\,  {c_2^+} \big\}$.
\end{proof}

\begin{corollary}\label{corconh5}{\rm [Uniform convexity in the theory including terms up to order $O(h^5)$]} For sufficiently small values of the thickness $h$ such that  $h|\kappa_1|<\frac{1}{2}$ and $h|\kappa_2|<\frac{1}{2}$
	and for constitutive coefficients  such that $\mu>0, \,\mu_{\rm c}>0$, $2\,\lambda+\mu> 0$, $b_1>0$, $b_2>0$ and $b_3>0$,   the  energy density
	\begin{align}W(\mathcal{E}_{m,s}, \mathcal{K}_{e,s})=W_{\mathrm{memb}}\big(  \mathcal{E}_{m,s} \big)+W_{\mathrm{memb,bend}}\big(  \mathcal{E}_{m,s} ,\,  \mathcal{K}_{e,s} \big)+W_{\mathrm{bend,curv}}\big(  \mathcal{K}_{e,s}    \big)
	\end{align}
	is uniformly convex in $(\mathcal{E}_{m,s}, \mathcal{K}_{e,s})$, i.e., there exists a constant $a_1^+>0$ such that
	\begin{align}
	D^2\,W(\mathcal{E}_{m,s},\mathcal{K}_{e,s}).\,[(H_1,H_2),(H_1,H_2)]\geq a_1^+(\lVert H_1\rVert^2+\lVert H_2\rVert^2) \qquad \forall\, H_1,H_2\in \mathbb{R}^{3\times 3}.
	\end{align}
\end{corollary}
\begin{proof}
	For a bilinear expression   $W(\mathcal{E}_{m,s}, \mathcal{K}_{e,s})$    in terms of $\mathcal{E}_{m,s}$ and $\mathcal{K}_{e,s}$, the second derivative  with respect
	to these argument variables coincides with the function itself, modulo a scalar multiplication.
	We will prove this known fact  only for two terms of the energy and we show that
\begin{align}\label{dervu}
& D^2(\lVert \sym\, \mathcal{E}_{m,s} \rVert^2).\,[(H_1,H_2),(H_1,H_2)]=2\,\lVert \sym\, H_1 \rVert^2
\qquad \text{and}\\
& D^2(\bigl\langle \sym\, \mathcal{E}_{m,s},\sym(\mathcal{E}_{m,s}{\rm B}_{y_0}+{\rm C}_{y_0}\, \mathcal{K}_{e,s})\bigr\rangle).\,[(H_1,H_2),(H_1,H_2)]=2\,\bigl\langle \sym\, H_1,\sym(H_1\,{\rm B}_{y_0}+{\rm C}_{y_0}\, H_2)\bigr\rangle.\notag
\end{align}
	Indeed, on the one hand, we have $D_F(\lVert F \rVert^2).\, H=2\,\langle\, F, H\,\rangle$, $\bigl\langle D^2_F(\lVert F \rVert^2).\, H,H\bigr\rangle=2\,\lVert  H\rVert^2$. Useful in our calculation is  that
	\begin{align}
		D^2\,W(\mathcal{E}_{m,s}, \mathcal{K}_{e,s}).\,[(H_1,H_2),(H_1,H_2)]=\,& D^2_{\mathcal{E}_{m,s}, \mathcal{E}_{m,s}}\,W(\mathcal{E}_{m,s}, \mathcal{K}_{e,s}).\,(H_1,H_1)\notag\\&+2\,  D_{ \mathcal{K}_{e,s}}[D_{\mathcal{E}_{m,s}}\,W(\mathcal{E}_{m,s}, \mathcal{K}_{e,s}).(H_1,H_1)].(H_2,H_2)\notag\\ &+ D^2_{\mathcal{K}_{e,s}, \mathcal{K}_{e,s}}\,W(\mathcal{E}_{m,s}, \mathcal{K}_{e,s}).\,(H_2,H_2)\notag.
	\end{align}
Since $\sym:\mathbb{R}^{3\times 3}\to {\rm Sym}(3)$ is a linear operator, we obtain
\begin{align}
&D^2(\lVert \sym\, \mathcal{E}_{m,s} \rVert^2).\,[(H_1,H_2),(H_1,H_2)]= D^2_{\mathcal{E}_{m,s}, \mathcal{E}_{m,s}}\,(\lVert \sym\, \mathcal{E}_{m,s} \rVert^2).\,(H_1,H_1)=2\,\lVert \sym\, H_1 \rVert^2,
\end{align}
which proves \eqref{dervu}$_1$.
 On the other hand,
	it holds
	\begin{align}
 D^2_{\mathcal{E}_{m,s}, \mathcal{E}_{m,s}}(\bigl\langle &\sym\, \mathcal{E}_{m,s},\sym(\mathcal{E}_{m,s}{\rm B}_{y_0}+{\rm C}_{y_0}\, \mathcal{K}_{e,s})\bigr\rangle).\,(H_1,H_1)=2\,
\bigl\langle \sym\, H_1,H_1{\rm B}_{y_0}\bigr\rangle,\notag\\
 D^2_{\mathcal{K}_{e,s}, \mathcal{K}_{e,s}}(\bigl\langle &\sym\, \mathcal{E}_{m,s},\sym(\mathcal{E}_{m,s}{\rm B}_{y_0}+{\rm C}_{y_0}\, \mathcal{K}_{e,s})\bigr\rangle).\,(H_2,H_2)=0,\\
D_{ \mathcal{K}_{e,s}}[D_{\mathcal{E}_{m,s}}\,(\bigl\langle &\sym\, \mathcal{E}_{m,s},\sym(\mathcal{E}_{m,s}{\rm B}_{y_0}+{\rm C}_{y_0}\, \mathcal{K}_{e,s})\bigr\rangle).(H_1,H_1)].(H_2,H_2)=\bigl\langle \sym\, H_1,{\rm C}_{y_0}\, H_2\bigr\rangle.\notag
\end{align}
	Therefore
\begin{align}
&D^2(\bigl\langle \sym\, \mathcal{E}_{m,s},\sym(\mathcal{E}_{m,s}{\rm B}_{y_0}+{\rm C}_{y_0}\, \mathcal{K}_{e,s})\bigr\rangle).\,[(H_1,H_2),(H_1,H_2)]=2\,
\bigl\langle \sym\, H_1,\sym(H_1{\rm B}_{y_0}+{\rm C}_{y_0}\, H_2)\bigr\rangle,
\end{align} which proves \eqref{dervu}$_2$.
In conclusion, after making similar calculations as above for the other terms appearing in the expression of $W(\mathcal{E}_{m,s},\mathcal{K}_{e,s})$, we obtain
\begin{align}
 D^2\,W(\mathcal{E}_{m,s},\mathcal{K}_{e,s}).\,[(H_1,H_2),(H_1,H_2)]=2\,W(H_1,H_2).
\end{align}
	Bounding the function $W(\mathcal{E}_{m,s},\mathcal{K}_{e,s})$ for all $\mathcal{E}_{m,s},\mathcal{K}_{e,s}\in \mathbb{R}^{3\times 3}$ away from zero amounts therefore to showing that
	$D^2W(\mathcal{E}_{m,s}, \mathcal{K}_{e,s})$
	is positive definite. Hence, the coercivity  of $W(\mathcal{E}_{m,s}, \mathcal{K}_{e,s})$ expressed by Proposition \ref{propcoerh5} implies uniform convexity in the chosen
	variables.
\end{proof}

\subsection{The existence result in the theory of order $O(h^5)$}

In this section, we prove the first main result of our paper.
The  admissible set $\mathcal{A}$ of solutions is defined by
\begin{equation}\label{21}
\mathcal{A}=\big\{(m,\overline{Q}_{e,s})\in{\rm H}^1(\omega, \mathbb{R}^3)\times{\rm H}^1(\omega, {\rm SO}(3))\,\,\big|\,\,\,  m\big|_{ \gamma_d}=m^*, \,\,\overline{Q}_{e,s}\big|_{ \gamma_d}=\overline{Q}_{e,s}^* \big\},
\end{equation}
where the boundary conditions are to be understood in the sense of traces.
\begin{theorem}\label{th1}{\rm [Existence result for the theory including terms up to order $O(h^5)$]}
	Assume that the external loads satisfy the conditions
	\begin{equation}\label{24}
	{f}\in\textrm{\rm L}^2(\omega,\mathbb{R}^3),\qquad  t\in \textrm{\rm L}^2(\gamma_t,\mathbb{R}^3),
	\end{equation}
	and the boundary data satisfy the conditions
	\begin{equation}\label{25}
	{m}^*\in{\rm H}^1(\omega ,\mathbb{R}^3),\qquad \overline{Q}_{e,s}^*\in{\rm H}^1(\omega, {\rm SO}(3)).
	\end{equation}
	Assume that the following conditions concerning the initial configuration are satisfied: $\,y_0:\omega\subset\mathbb{R}^2\rightarrow\mathbb{R}^3$ is a continuous injective mapping and
	\begin{align}\label{26}
{y}_0&\in{\rm H}^1(\omega ,\mathbb{R}^3),\qquad   {Q}_{0}(0)\in{\rm H}^1(\omega, {\rm SO}(3)),\qquad
	\nabla_x\Theta(0)\in {\rm L}^\infty(\omega ,\mathbb{R}^{3\times 3}),\qquad \det[\nabla_x\Theta(0)] \geq\, a_0 >0\,,
\end{align}
	where $a_0$ is a constant.
	Then, for sufficiently small values of the thickness $h$ such that  $h|\kappa_1|<\frac{1}{2}$ and $h|\kappa_2|<\frac{1}{2}$
	and for constitutive coefficients  such that $\mu>0, \,\mu_{\rm c}>0$, $2\,\lambda+\mu> 0$, $b_1>0$, $b_2>0$ and $b_3>0$, the minimization problem \eqref{e89}--\eqref{boundary} admits at least one minimizing solution pair
	$(m,\overline{Q}_{e,s})\in  \mathcal{A}$.
\end{theorem}
\begin{proof}
We employ the direct methods of the calculus of variations, similar to \cite{NeffBirsan13,neff2014existence,Neff_Habil04}. However, in comparison to \cite{NeffBirsan13}, due to the fact that we use only matrix notation, some steps   are shortened.   In Proposition \ref{propcoerh5}  and Corollary \ref{corconh5}, we have shown  that the strain energy density $W(\mathcal{E}_{m,s}, \mathcal{K}_{e,s})$ is  a quadratic convex and coercive function of $(\mathcal{E}_{m,s}, \mathcal{K}_{e,s})$.

The hypothesis \eqref{24} and the boundedness of $\,\,\Pi_{S^0}\,$ and $ \,\Pi_{\partial S^0_f}\,$ imply that there exists a constant $C>0$ such that\footnote{ {By $C$ and $C_i$, $i\in\mathbb{N}$, we will denote (positive) constants that may vary from estimate to estimate but will remain independent of $m$, $\nabla m$ and $\overline{Q}_{e,s}$.}}
\begin{equation*}
\begin{array}{l}
|\overline{\Pi}(m,\overline{Q}_{e,s})|\,\leq\,
C\,\Big(\lVert m-y_0\rVert_{{\rm L}^2(\omega)} + \lVert m-y_0\rVert_{{\rm L}^2(\gamma_t)} +\lVert \overline{Q}_{e,s}\rVert_{{\rm L}^2(\omega)}\Big)\, \quad \forall \, (m,\overline{Q}_{e,s})\in  {\rm H}^1(\omega, \mathbb{R}^3)\times{\rm H}^1(\omega, {\rm SO}(3)).
\end{array}
\end{equation*}

We have $
\label{R3}\lVert \overline{Q}_{e,s}\rVert ^2=\text{tr}(\overline{Q}_{e,s}  \overline{Q}_{e,s} ^T)=\tr(\id_3)=3, \  \,\forall\,\overline{Q}_{e,s}\in {\rm SO}(3).
$
 Hence,  there exists a constant $C>0$ such that
\begin{equation}\label{27}
|\,\overline{\Pi}(m,\overline{Q}_{e,s})\,|\,\leq\, \,\,C\,\big(\,\lVert m\rVert_{{\rm H}^1(\omega)}+1\big),\quad\forall\,(m,\overline{Q}_{e,s})\in  {\rm H}^1(\omega, \mathbb{R}^3)\times{\rm H}^1(\omega, {\rm SO}(3)).
\end{equation}
Considering
\begin{equation}
\overline{R}_s(x_1,x_2) =\overline{Q}_{e,s}(x_1,x_2)\,Q_0(x_1,x_2,0)\in{\rm SO}(3),
\end{equation}
we observe that
\begin{align}
\mathcal{E}_{m,s}  \,=& \,   Q_0  [\overline{R}_s^T (\nabla  m|\overline{Q}_{e,s}\nabla_x\Theta(0)\, e_3)- Q_0^T (\nabla  y_0|n_0)][\nabla_x \Theta(0)]^{-1}=\,Q_0( \overline{R}_s^T\,\nabla m -Q_0^T\nabla y_0 |0)[\nabla_x \Theta(0)]^{-1}.
\end{align}
The lifted quantity
$
\widehat{\rm I}_{y_0}\:\,=\,({\nabla  y_0}|n_0)^T({\nabla  y_0}|n_0)\in {\rm Sym}(3)$ is positive definite and  also it's inverse is positive definite. Using the above relation we obtain
\begin{align}\label{Elambda}
\lVert \mathcal{E}_{m,s}\rVert ^2& =\bigl\langle\,  Q_0^T\,Q_0( \overline{R}_s^T\,\nabla m -Q_0^T\nabla y_0 |0), ( \overline{R}_s^T\,\nabla m -Q_0^T\nabla y_0 |0)\widehat{\rm I}^{-1}_{y_0} \bigr\rangle \notag\\
& =\bigl\langle\,  \widehat{\rm I}^{-1}_{y_0}( \overline{R}_s^T\,\nabla m -Q_0^T\nabla y_0 |0)^T, ( \overline{R}_s^T\,\nabla m -Q_0^T\nabla y_0 |0)^T \bigr\rangle
\geq\, \lambda_0^2\,\lVert ( \overline{R}_s^T\,\nabla m -Q_0^T\nabla y_0 |0)\rVert ^2,
\end{align}
where $\lambda_0$ is the smallest eigenvalue of the positive definite matrix $\widehat{\rm I}^{-1}_{y_0}$.
Similarly, we deduce that
\begin{align}\lVert \mathcal{K}_{e,s} \rVert ^2
 &=\bigl\langle (\mathrm{axl}(\overline{Q}_{e,s}^T\,\partial_{x_1} \overline{Q}_{e,s})\,|\, \mathrm{axl}(\overline{Q}_{e,s}^T\,\partial_{x_2} \overline{Q}_{e,s})\,|0), (\mathrm{axl}(\overline{Q}_{e,s}^T\,\partial_{x_1} \overline{Q}_{e,s})\,|\, \mathrm{axl}(\overline{Q}_{e,s}^T\,\partial_{x_2} \overline{Q}_{e,s})\,|0)\,\widehat{\rm I}^{-1}_{y_0}\bigr\rangle\notag\\
 &=\bigl\langle \widehat{\rm I}^{-1}_{y_0}(\mathrm{axl}(\overline{Q}_{e,s}^T\,\partial_{x_1} \overline{Q}_{e,s})\,|\, \mathrm{axl}(\overline{Q}_{e,s}^T\,\partial_{x_2} \overline{Q}_{e,s})\,|0)^T, (\mathrm{axl}(\overline{Q}_{e,s}^T\,\partial_{x_1} \overline{Q}_{e,s})\,|\, \mathrm{axl}(\overline{Q}_{e,s}^T\,\partial_{x_2} \overline{Q}_{e,s})\,|0)^T\bigr\rangle\\
 &\geq \, \lambda_0^2\, \lVert (\mathrm{axl}(\overline{Q}_{e,s}^T\,\partial_{x_1} \overline{Q}_{e,s})\,|\, \mathrm{axl}(\overline{Q}_{e,s}^T\,\partial_{x_2} \overline{Q}_{e,s}))\rVert ^2.\notag
\end{align}
  From \eqref{Elambda} we have
\begin{align}
\lVert \mathcal{E}_{m,s}\rVert ^2&
\geq\, \lambda_0^2\Big[\lVert \overline{R}_s^T\nabla m\rVert ^2 -2\,\bigl\langle \overline{R}_s^T\,\nabla m ,\,Q_0^T\nabla y_0 \bigr\rangle + \lVert  Q_0^T\nabla y_0 \rVert ^2\Big].
\end{align}
Since $\lVert \overline{R}_s^T\nabla m\rVert ^2=\lVert\nabla m\rVert ^2$ and $\lVert  Q_0^T\nabla y_0 \rVert ^2=\lVert  \nabla y_0 \rVert ^2$,   after  integrating over $\omega$, using \eqref{R3}, the Cauchy--Schwarz inequality  and  the hypothesis upon $y_0$,   gives us the estimate
\begin{align}\lVert \mathcal{E}_{m,s}\rVert ^2_{{\rm L}^2(\omega)} &\geq\,
\lambda_0^2\Big[\lVert \nabla m\rVert ^2_{{\rm L}^2(\omega)} -2\,\lVert \nabla m \rVert_{{\rm L}^2(\omega)}\lVert \nabla y_0\rVert_{{\rm L}^2(\omega)}+ \lVert \nabla y_0\rVert ^2_{{\rm L}^2(\omega)}\Big]\notag\\
&\geq\,
\lambda_0^2\lVert \nabla m\rVert ^2_{{\rm L}^2(\omega)}- C_1\,\lVert \nabla m \rVert_{{\rm L}^2(\omega)}+ C_2,
\end{align}
for some positive constants ${C}_1>0$, ${C}_2>0$.

 By virtue of the coercivity of the internal energy and \eqref{26}, \eqref{27} and \eqref{Elambda}, the functional $I(m,\overline{Q}_{e,s})$ is bounded from below
\begin{align}\label{IC}
I(m,\overline{Q}_{e,s})&\geq\, C_1 \!\!\int_\omega \lVert \mathcal{E}_{m,s}\rVert ^2\,\det[\nabla_x\Theta(0)]\,  {{\rm d}a} - \overline{\Pi}(m,\overline{Q}_{e,s}) \geq\, C_2\,a_0 \lVert \mathcal{E}_{m,s}\rVert ^2_{{\rm L}^2(\omega)}  - C_3\,\big(\,\lVert m\rVert_{{\rm H}^1(\omega)}+1\big)\notag\\&\geq\,  C_4\, \lVert \nabla m\rVert ^2_{{\rm L}^2(\omega)} - C_5\,\lVert m\rVert_{{\rm H}^1(\omega)}-{C}_6\, \quad \forall \, (m,\overline{Q}_{e,s})\in  {\rm H}^1(\omega, \mathbb{R}^3)\times{\rm H}^1(\omega, {\rm SO}(3)),
\end{align}
with $C_i>0$, $i=1,2,...,6$.
We also obtain, applying the Poincar\'e--inequality, that  there exists  a constant $C>0$ such that
\begin{align}
\lVert \nabla m\rVert ^2_{{\rm L}^2(\omega)}&\geq\, (\lVert \nabla (m-m^*)\rVert_{{\rm L}^2(\omega)}-\lVert \nabla m^*\rVert_{{\rm L}^2(\omega)})^2\notag\\&\geq\, C\,\lVert  m-m^*\rVert _{{\rm H}^1(\omega)}^2-2\,\lVert  m-m^*\rVert _{{\rm H}^1(\omega)} \lVert \nabla m^*\rVert_{{\rm L}^2(\omega)}+\lVert \nabla m^*\rVert_{{\rm L}^2(\omega)}^2 \notag\\&\geq\, C\,\lVert  m-m^*\rVert _{{\rm H}^1(\omega)}^2-\frac{1}{\varepsilon}\,\lVert  m-m^*\rVert ^2_{{\rm H}^1(\omega)}-\varepsilon \lVert \nabla m^*\rVert ^2_{{\rm L}^2(\omega)}+\lVert \nabla m^*\rVert_{{\rm L}^2(\omega)}^2 \,\qquad \forall \, \varepsilon>0.
\end{align}
Therefore, by choosing $\varepsilon>0$ small enough,  \eqref{IC} ensures the existence of  constants $C_1>0$ and $C {_2}\in\mathbb{R}$ such that
\begin{align}\label{36}
I(m,\overline{Q}_{e,s})&\geq\, C_1\lVert  m-m^*\rVert_{{\rm H}^1(\omega)}^2 +C_2
\, \quad \forall \, (m,\overline{Q}_{e,s})\in  {\rm H}^1(\omega, \mathbb{R}^3)\times{\rm H}^1(\omega, {\rm SO}(3)),
\end{align}
i.e.,  the functional $I(m,\overline{Q}_{e,s})$ is bounded from below on $\mathcal{A}$.

Hence, there exists an infimizing sequence $\big\{(m_k,\overline{Q}_{k})\big\}_{k=1}^\infty$ in $  \mathcal{A}$, such that
\begin{equation}\label{37}
\lim_{k\rightarrow \infty} I(m_k,\overline{Q}_{k}) = \,\inf\, \big\{I(m,\overline{Q}_{e,s})\, \big|\,  (m,\overline{Q}_{e,s})\in  \mathcal{A}\big\}.
\end{equation}
Since we have $I(m^*,\overline{Q}_{e,s}^*)<\infty$,  in view of the conditions \eqref{25}, the infimizing sequence $\big\{(m_k,\overline{Q}_{k})\big\}_{k=1}^\infty$ can be chosen such that
\begin{equation}\label{38}
I(m_k,\overline{Q}_{k})\,\leq \,I(m^*,\overline{Q}_{e,s}^*)\,< \infty\,, \qquad \forall\,k\geq\, 1.
\end{equation}
Taking into account \eqref{36} and \eqref{38} we see that the sequence $\big\{m_k \big\}_{k=1}^\infty$ is bounded in ${\rm H}^1(\omega,\mathbb{R}^3)$. Then, we can extract a subsequence of $\big\{m_k \big\}_{k=1}^\infty$  (not relabeled) which converges weakly in  ${\rm H}^1(\omega,\mathbb{R}^3)$ and moreover, according to Rellich's selection principle, it converges strongly in ${\rm L}^2(\omega,\mathbb{R}^{3})$, i.e., ~there exists an element $\widehat{m}\in{\rm H}^1(\omega,\mathbb{R}^3)$ such that
\begin{equation}\label{39}
m_k  \rightharpoonup \widehat{m} \quad\mathrm{in}\quad {\rm H}^1(\omega, \mathbb{R}^3),\qquad \mathrm{and}\qquad m_k \rightarrow\widehat{ m} \quad\mathrm{in}\quad {\rm L}^2(\omega, \mathbb{R}^3).
\end{equation}
Corresponding to the fields $(m_k,\overline{Q}_{k})$ we consider the strain measures   $\mathcal{E}_{m,s}^{(k)},\,\mathcal{K}_{e,s}^{(k)}\in {\rm L}^2(\omega,\mathbb{R}^{3\times 3})$. From the coercivity of the internal energy, \eqref{27} and \eqref{38} we get
$$ C_1\,\lVert \mathcal{K}_{e,s}^{(k)}\rVert ^2_{{\rm L}^2(\omega)} \leq
\int_\omega W(\mathcal{E}_{m,s}^{(k)},\mathcal{K}_{e,s}^{(k)})\,\det[\nabla_x\Theta(0)] \,  {{\rm d}a}\leq I(m^*,\overline{Q}_{e,s}^*)+ C_2\,\big(\,\lVert m_k\rVert_{{\rm H}^1(\omega)}+1\big),
$$
where $C_1,C_2$ are positive constants.

Since $\big\{ m_k \big\}_{k=1}^\infty$ is bounded in ${\rm H}^1(\omega,\mathbb{R}^3)$, it follows from the last inequalities that $\big\{ \mathcal{K}_{e,s}^{(k)} \big\}_{k=1}^\infty$ is bounded in ${\rm L}^2(\omega,\mathbb{R}^{3\times 3})$.

  For tensor fields $P$ with rows in ${\rm H}({\rm curl}\,; \Omega)$, i.e., 
$
P=\begin{pmatrix}
P^T.e_1\,|\,
P^T.e_2\,|\,
P^T\, e_3
\end{pmatrix}^T$ with $(P^T.e_i)^T\in {\rm H}({\rm curl}\,; \Omega)$, $i=1,2,3$,
we define
$
{\rm Curl}\,P:=\begin{pmatrix}
{\rm curl}\, (P^T.e_1)^T\,|\,
{\rm curl}\, (P^T.e_2)^T\,|\,
{\rm curl}\, (P^T\, e_3)^T
\end{pmatrix}^T
.$
Since $\big\{\mathcal{K}_{e,s}^{(k)}\}_{k=1}^\infty$ is bounded, so is   $\big\{\text{axl}(\overline{Q}_{k}^{T}\partial_{x_\alpha} \overline{Q}_{k})\big\}_{k=1}^\infty$, $\alpha=1,2$, in ${\rm L}^2(\omega,\mathbb{R}^{3})$ and it follows that $\overline{Q}_{k}^T\,\Curl \,\overline{Q}_{k}$ is bounded.
 Indeed, using the
so-called \textit{{wryness tensor}} (second order tensor) \cite{Neff_curl06,Pietraszkiewicz04}
\begin{align}
\Gamma_k&:= \Big(\mathrm{axl}(\overline{Q}_{k}^T\,\partial_{x_1} \overline{Q}_{k})\,|\, \mathrm{axl}(\overline{Q}_{k}^T\,\partial_{x_2} \overline{Q}_{k})\,|\,0\,\Big)\in \mathbb{R}^{3\times 3},
\end{align}
we have (see \cite{Neff_curl06})  the following close relationship (Nye's formula) between the wryness tensor
and the dislocation density tensor
\begin{align}\label{Nye1}
\alpha_k\,:=\overline{Q}_{k}^T\,\Curl\, \overline{Q}_{k}=\,-\Gamma_k^T+\tr(\Gamma_k)\, \id_3\,, \qquad\textrm{or equivalently},\qquad \Gamma_k\,=\,-\alpha_k^T+\frac{1}{2}\tr(\alpha_
k)\, \id_3\,,
\end{align}
because $\overline{Q}_{k}=\overline{Q}_{k}(x_1,x_2)$. Hence, $\big\{\text{axl}(\overline{Q}_{k}^{T}\partial_{x_\alpha} \overline{Q}_{k})\big\}_{k=1}^\infty$ is bounded if and only if $\overline{Q}_{k}^T\,\Curl \,\overline{Q}_{k}$ is bounded.
Writing $\Curl \,\overline{Q}_{k}=\overline{Q}_{k}\,\overline{Q}_{k}^T\,\Curl \,\overline{Q}_{k}$ and using $\lVert \overline{Q}_{k}\rVert ^2=3$, we deduce that the boundedness of $\overline{Q}_{k}^T\,\Curl \,\overline{Q}_{k}$  implies that  $\Curl \,\overline{Q}_{k}$ is bounded.   Since the $\Curl$-operator bounds the gradient operator in ${\rm SO}(3)$, see \cite{Neff_curl06}, it follows  that
$ \big\{ \partial_{x_\alpha} \overline{Q}_{k}\big\}_{k=1}^\infty$ is bounded in ${\rm L}^2(\omega,\mathbb{R}^{3\times 3})$, for $\alpha=1,2$. Since $\overline{Q}_{k}\in {\rm SO}(3)$ we have $\lVert \overline{Q}_{k}\rVert ^2=3$ and thus we can infer that the sequence $ \big\{ \overline{Q}_{k}\big\}_{k=1}^\infty$
is bounded in ${\rm H}^1(\omega,\mathbb{R}^{3\times 3})$. Hence, there exists a subsequence of $ \big\{ \overline{Q}_{k}\big\}_{k=1}^\infty$ (not relabeled) and an element $\widehat{\overline{Q}}_{e,s}\in {\rm H}^1(\omega,\mathbb{R}^{3\times 3})$ with
\begin{equation}\label{40}
\overline{Q}_{k} \rightharpoonup       \widehat{\overline{Q}}_{e,s}    \quad\mathrm{in}\quad {\rm H}^1(\omega, \mathbb{R}^{3\times3}) , \qquad\mathrm{and}\qquad         \overline{Q}_{k}  \rightarrow   \widehat{\overline{Q}}_{e,s} \quad\mathrm{in}\quad {\rm L}^2(\omega, \mathbb{R}^{3\times3}).
\end{equation}
Since $ \overline{Q}_{k} \in {\rm SO}(3)$ we have
$$\lVert  \overline{Q}_{k}  \widehat{\overline{Q}}_{e,s} ^T -\id_3\rVert_{{\rm L}^2(\omega)}= \lVert  \overline{Q}_{k}  ( \widehat{\overline{Q}}_{e,s} ^T -\overline{Q}_{k}^T)\rVert_{{\rm L}^2(\omega)}=  \lVert   \widehat{\overline{Q}}_{e,s}  - \overline{Q}_{k} \rVert_{{\rm L}^2(\omega)} \rightarrow 0,
$$
i.e.,  $\overline{Q}_{k}\widehat{\overline{Q}}_{e,s}^T\rightarrow\id_3$ in ${\rm L}^2(\omega,\mathbb{R}^{3\times 3})$. On the other hand, we can write
$$\lVert \overline{Q}_{k} \widehat{\overline{Q}}_{e,s}^T -\widehat{\overline{Q}}_{e,s} \widehat{\overline{Q}}_{e,s}^T\rVert_{{\rm L}^1(\omega)}= \lVert  ( \overline{Q}_{k} -\widehat{\overline{Q}}_{e,s})\widehat{\overline{Q}}_{e,s}^T\rVert_{L^1(\omega)} \leq 3  \, \lVert    \overline{Q}_{k}  -\widehat{\overline{Q}}_{e,s} \rVert_{{\rm L}^2(\omega)}\,  \lVert   \widehat{\overline{Q}}_{e,s} \rVert_{{\rm L}^2(\omega)} \rightarrow 0,
$$
which means that $\overline{Q}_{k} \widehat{\overline{Q}}_{e,s}^T\rightarrow\widehat{\overline{Q}}_{e,s}\widehat{\overline{Q}}_{e,s}^T$ in ${\rm L}^1(\omega,\mathbb{R}^{3\times 3})$. Consequently, we find $\widehat{\overline{Q}}_{e,s}\widehat{\overline{Q}}_{e,s}^T=\id_3$ so that $\widehat{\overline{Q}}_{e,s}$ belongs to ${\rm H}^1(\omega,{\rm SO}(3))$.

By virtue of the relations  $(m_k,\overline{Q}_{k})\in \mathcal{A}$ and \eqref{39}, \eqref{40},   we derive that $\widehat{m}={m}^*$ on $\gamma_d$ and $\widehat{\overline{Q}}_{e,s}=\overline{Q}_{e,s}^*$ on $\gamma_d\,$ in the sense of traces. Hence, we obtain that the limit pair satisfies $(\widehat{m},\widehat{\overline{Q}}_{e,s})\in \mathcal{A}$.

Let us next construct the limit strain and curvature measures
\begin{align}\label{e55}
\widehat{\mathcal{E}}_{m,s}  :&=\,    \widehat{\overline{Q}}_{e,s}^T  (\nabla   {\widehat{m}}|\widehat{\overline{Q}}_{e,s}\nabla_x\Theta(0)\, e_3)[\nabla_x \Theta(0)]^{-1}-\id_3=Q_0( \widehat{\overline{R}}_s^T\,\nabla  {\widehat{m}} -Q_0^T\nabla y_0 |0)[\nabla_x \Theta(0)]^{-1} , \notag \\
\widehat{\mathcal{K}}_{e,s}  :&=\,  (\mathrm{axl}(\widehat{\overline{Q}}_{e,s}^T\,\partial_{x_1} \widehat{\overline{Q}}_{e,s})\,|\, \mathrm{axl}(\widehat{\overline{Q}}_{e,s}^T\,\partial_{x_2} \widehat{\overline{Q}}_{e,s})\,|0)[\nabla_x \Theta(0)]^{-1}\\&= Q_0\big(\,\text{axl}(\widehat{\overline{R}}_s^T\partial_{x_1} \widehat{\overline{R}}_s)\!-\!\text{axl}(Q^{T}_0\partial_{x_1} Q_0) \,\,\big|\,\, \text{axl}(\widehat{\overline{R}}_s^T\partial_{x_2} \widehat{\overline{R}}_s)\! -\!\text{axl}(Q^{T}_0\partial_{x_2} Q_0)\,\,\big|\, \,0\,\big)[\nabla_x \Theta(0)]^{-1},\notag
\end{align}
where
\begin{equation}
\widehat{\overline{R}}_s(x_1,x_2) :=\widehat{\overline{Q}}_{e,s}(x_1,x_2)\,Q_0(x_1,x_2,0)\in{\rm SO}(3).
\end{equation}

As shown above, the sequence $ \big\{ m_k\big\}_{k=1}^\infty$ is bounded in ${\rm H}^1(\omega,\mathbb{R}^{3 })$. It follows that $ \big\{(\nabla m_k|0)\big\}_{k=1}^\infty$ is bounded in ${\rm L}^2(\omega,\mathbb{R}^{3\times 3})$. We define
\begin{align}
\overline{R}_k :=\overline{Q}_k\,Q_0\in{\rm SO}(3).
\end{align}
Then, the sequence $ \big\{ \overline{R}_k^T(\nabla m_k|0)\big\}_{k=1}^\infty$ is bounded in ${\rm L}^2(\omega,\mathbb{R}^{3\times 3 })$, since $ \overline{R}_k\in {\rm SO}(3)$. Consequently, there exists a subsequence (not relabeled) and an element $\xi\in {\rm L}^2(\omega,\mathbb{R}^{3\times 3})$ such that
\begin{equation}\label{43}
\overline{R}_k^T(\nabla m_k|0) \,\,\rightharpoonup \,\, \xi\qquad\text{in} \quad {\rm L}^2(\omega,\mathbb{R}^{3\times 3}).
\end{equation}
On the other hand, let ${\Phi}\in  {\rm C}_0^\infty(\omega,\mathbb{R}^{3\times 3})$ be an arbitrary test function. Then, using the properties of the scalar product  we deduce
\begin{align}
\displaystyle{\int_\omega}\bigl\langle\overline{R}_k^T(\nabla m_k|0) - \widehat{\overline{R}}_s^T(\nabla \widehat{m}|0),\,\Phi \bigr\rangle\, {{\rm d}a}&=
\displaystyle{\int_\omega} \bigl\langle \widehat{\overline{R}}_s^T \big( (\nabla m_k|0)-(\nabla \widehat{m}|0)\big) ,{\Phi}\bigr\rangle \, {{\rm d}a}
+\displaystyle{\int_\omega}  \bigl\langle\big( \overline{R}_k^T- \widehat{\overline{R}}_s^T\big) (\nabla m_k|0) ,{\Phi} \bigr\rangle\, {{\rm d}a}\notag
\\
&   =\displaystyle{\int_\omega}  \bigl\langle (\nabla m_k|0)-(\nabla \widehat{m}|0),\widehat{\overline{R}}_s{\Phi} \bigr\rangle\,  {{\rm d}a}  +
\displaystyle{\int_\omega} \!\! \bigl\langle \overline{R}_{k}\!\!-\! \widehat{\overline{R}}_s\,,(\nabla m_k|0)\,{\Phi}^T \bigr\rangle  {{\rm d}a}
\\
& \leq\lVert \overline{R}_{k}\!-\! \widehat{\overline{R}}_s\rVert_{{\rm L}^2(\omega)}\lVert (\nabla m_k|0)\,{\Phi}^T\rVert_{{\rm L}^2(\omega)} \!+\!\! \displaystyle{\int_\omega}  \bigl\langle(\nabla m_k|0)-(\nabla \widehat{m}|0),\widehat{\overline{R}}_s{\Phi} \bigr\rangle\,  {{\rm d}a\, .}\notag
\end{align}
 {S}ince the relations \eqref{39}, \eqref{40} and $\widehat{\overline{R}}_s{\Phi}\in{\rm L}^2(\omega,\mathbb{R}^{3\times 3})$ hold, and $\lVert(\nabla m_k|0)\,{\Phi}^T\rVert$ is bounded{, } we get
\begin{equation}\label{44}
\displaystyle{\int_\omega}\bigl\langle\overline{R}_k^T (\nabla m_k|0) ,{\Phi} \bigr\rangle\, {{\rm d}a} \rightarrow
\displaystyle{\int_\omega}\bigl\langle \widehat{\overline{R}}_s^T (\nabla \widehat{m}|0),{\Phi} \bigr\rangle\, {{\rm d}a},\quad\forall\, {\Phi}\in  {\rm C}_0^\infty(\omega,\mathbb{R}^{3\times 3}).
\end{equation}
By comparison of \eqref{43} and \eqref{44} we find ${\xi}= \widehat{\overline{R}}_s^{T}(\nabla \widehat{m}|0)$, which means that
$\overline{R}_k^T (\nabla m_k|0) \rightharpoonup
\widehat{\overline{R}}_s^T (\nabla \widehat{m}|0)$ in $ {\rm L}^2(\omega,\mathbb{R}^{3\times 3})$,
or equivalently
\begin{equation}\label{45}
\overline{R}_k^T (\nabla m_k|0)  - Q^{T}_0(\nabla y_0|0) \quad \rightharpoonup\quad
\overline{R}_k^T (\nabla \widehat{m}|0)  - Q^{T}_0(\nabla y_0|0)\quad\mathrm{in}\quad {\rm L}^2(\omega,\mathbb{R}^{3\times 3}).
\end{equation}
Taking into account  the hypotheses, we obtain from \eqref{45} that \begin{align}\label{convE}
\mathcal{E}_{m,s}^{(k)}  :&=Q_0( \overline{R}_k^T\,\nabla m_k -Q_0^T\nabla y_0 |0)[\nabla_x \Theta(0)]^{-1}\rightharpoonup\widehat{\mathcal{E}}_{m,s}
\end{align} in ${\rm L}^2(\omega,\mathbb{R}^{3})$.

We use now the fact that the sequence $ \big\{\text{axl}( \overline{R}_k^T\partial_{x_\alpha} \overline{R}_k)\big\}_{k=1}^\infty$, $\alpha=1,2$, is bounded in ${\rm L}^2(\omega,\mathbb{R}^{3 })$, since  we proved previously that $ \overline{R}_k^T\partial_{x_\alpha} \overline{R}_k$ is bounded in ${\rm L}^2(\omega,\mathbb{R}^{3 \times 3})$. Then, there exists a subsequence (not relabeled) and an element ${ \zeta}_\alpha\in {\rm L}^2(\omega,\mathbb{R}^{3})$, $\alpha=1,2$,  such that
\begin{equation}\label{46}
\text{axl}(\overline{R}_k^T\partial_{x_\alpha} \overline{R}_k) \rightharpoonup { \zeta}_\alpha\qquad\text{in} \quad {\rm L}^2(\omega,\mathbb{R}^{3}).
\end{equation}
On the other hand, for any test function ${{\phi}}\in  {\rm C}_0^\infty(\omega,\mathbb{R}^{3})$ we can write
\begin{align}
\displaystyle{\int_\omega}& \bigl\langle \text{axl}(\overline{R}_k^T \partial_{x_\alpha} \overline{R}_{k} - \widehat{\overline{R}}_s^T \partial_{x_\alpha} \widehat{\overline{R}}_s)  \,,\, {\phi}  \bigr\rangle_{\mathbb{R}^3}  \, {{\rm d}a}=\frac{1}{2}\,\displaystyle{\int_\omega} \bigl\langle \overline{R}_k^T \partial_{x_\alpha} \overline{R}_{k} - \widehat{\overline{R}}_s^T \partial_{x_\alpha} \widehat{\overline{R}}_s  \,,\,{\rm anti} ({\phi})  \bigr\rangle _{\mathbb{R}^{3\times 3}}   \, {{\rm d}a}
\notag\\&=\frac{1}{2}
\displaystyle{\int_\omega} \bigl\langle   \widehat{\overline{R}}_s^T \big(\partial_{x_\alpha}  \overline{R}_{k} - \partial_{x_\alpha} \widehat{\overline{R}}_s\big) \,,\, {\rm anti} ({\phi}) \bigr\rangle  _{\mathbb{R}^{3\times 3}}  \, {{\rm d}a} +
\frac{1}{2}\displaystyle{\int_\omega} \bigl\langle \big( \overline{R}_k^T- \widehat{\overline{R}}_s^T\big) \partial_{x_\alpha} \overline{R}_{k}  \,,\, {\rm anti} ({\phi}) \bigr\rangle_{\mathbb{R}^{3\times 3}}    \, {{\rm d}a}\\
&\leq
\frac{1}{2}\displaystyle{\int_\omega}
 \bigl\langle \partial_{x_\alpha} \overline{R}_{k} - \partial_{x_\alpha} \widehat{\overline{R}}_s  \,,\, \widehat{\overline{R}}_s\,{\rm anti} ({\phi}) \bigr\rangle _{\mathbb{R}^{3\times 3}}  \, {{\rm d}a}
+
\frac{1}{2}\lVert \overline{R}_{k}- \widehat{\overline{R}}_s\rVert_{{\rm L}^2(\omega)}\,\lVert  \partial_{x_\alpha} \overline{R}_k\, [{\rm anti} ({\phi})]^T \, \rVert_{{\rm L}^2(\omega)} \rightarrow 0,\notag
\end{align}
since $\widehat{\overline{R}}_s\, {\rm anti} ({\phi})\in{\rm L}^2(\omega,\mathbb{R}^{3\times 3})$, $\lVert \partial_{x_\alpha} \overline{R}_k \,[{\rm anti} ({\phi})]^T\rVert $ is bounded, and relations \eqref{40} hold. Consequently, we have
\begin{align}\displaystyle{\int_\omega} \bigl\langle \text{axl}(\overline{R}_k^T \partial_{x_\alpha} \overline{R}_{k})\,,\, \phi\bigr\rangle _{\mathbb{R}^{3}}  \, {{\rm d}a} \rightarrow
\displaystyle{\int_\omega} \bigl\langle  \text{axl}(\widehat{\overline{R}}_s^T \partial_{x_\alpha} \widehat{\overline{R}}_s)\,,\, \phi \bigr\rangle _{\mathbb{R}^{3}}  \, {{\rm d}a},\quad\forall \,\phi\in  {\rm C}_0^\infty(\omega,\mathbb{R}^{ {3}}), \end{align}
and by comparison with \eqref{46} we deduce that ${ \zeta}_\alpha=\text{axl}( \widehat{\overline{R}}_s{}^{T}\partial_{x_\alpha} \widehat{\overline{R}}_s)\,$, i.e., 
\begin{align}\text{axl}(\overline{R}_k^T \partial_{x_\alpha} \overline{R}_k) - \text{axl}(Q_0 \partial_{x_\alpha} Q_0)
\,\,\,\rightharpoonup \,\,\,
\text{axl}(\widehat{\overline{R}}_s{}^{T}\partial_{x_\alpha} \widehat{\overline{R}}_s)- \text{axl}(Q_0 \partial_{x_\alpha} Q_0)
\quad\text{in}\,\, {\rm L}^2(\omega,\mathbb{R}^{ {3}}),\end{align}
Hence, from  \eqref{26} we derive  the convergence
\begin{align}\label{convK}
\mathcal{K}_{e,s}^{(k)}  :&= Q_0\big(\,\text{axl}({\overline{R}}_k^T\partial_{x_1} {\overline{R}}_k)\!-\!\text{axl}(Q^{T}_0\partial_{x_1} Q_0) \,\,\big|\,\, \text{axl}({\overline{R}}_k^T\partial_{x_2}{\overline{R}}_k)\! -\!\text{axl}(Q^{T}_0\partial_{x_2} Q_0)\,\,\big|\, \,0\,\big)[\nabla_x \Theta(0)]^{-1} \rightharpoonup\widehat{\mathcal{K}}_{e,s}.
\end{align}

In the last step of the proof we use the convexity of the strain energy density $W$. In view of \eqref{convE} and \eqref{convK}, we have
\begin{equation}\label{47}
\int_\omega W(\widehat{\mathcal{E}}_{m,s},\widehat{\mathcal{K}}_{e,s})\,\det[\nabla_x\Theta(0)] \, {{\rm d}a}\,\leq \, \liminf_{n\to\infty}  \int_\omega W(\mathcal{E}_{m,s}^{(k)},\mathcal{K}_{e,s}^{(k)})\,\det[\nabla_x\Theta(0)] \, {{\rm d}a}.
\end{equation}
since $W$ is convex in $({\mathcal{E}}_{m,s},{\mathcal{K}}_{e,s})$. Taking into account the hypotheses \eqref{24}, the continuity of the load potential functions,    and the convergence relations \eqref{39}$_2$ and \eqref{40}$_2\,$, we deduce
\begin{equation}\label{48}
 {\overline{\Pi}}(\widehat{{m}}, \widehat{\overline{Q}}_{e,s})=  \lim_{n\to\infty}  {\overline{\Pi}}(m_k, \overline{Q}_k).
\end{equation}
From \eqref{47} and \eqref{48} we get
\begin{equation}\label{49}
I(\widehat{{m}},\widehat{\overline{Q}}_{e,s})\,\leq\,  \liminf_{n\to\infty} \, I(m_k,\overline{Q}_k)\,.
\end{equation}
Finally, the relations \eqref{37} and \eqref{49} show that
$$ I(\widehat{{m}},\widehat{\overline{Q}}_{e,s})\,=\,
\,\inf\, \big\{I({m},\overline{Q}_{e,s})\, \big|\,  ({m},\overline{Q}_{e,s})\in \mathcal{A}\big\}.
$$
Since $(\widehat{{m}},\widehat{\overline{Q}}_{e,s})\in\mathcal{A}$, we conclude that $(\widehat{{m}},\widehat{\overline{Q}}_{e,s})$ is a minimizing solution pair of our minimization problem.
\end{proof}
The boundary condition on $\overline{Q}_{e,s}$ is not essential in the proof of the above theorem, and that one can prove the existence of minimizers for the minimization problem over a larger admissible set:
\begin{corollary}\label{th1nD}{\rm [Existence result for the theory including terms up to order $O(h^5)$ without boundary condition on the microrotation field]}
Under the hypotheses of Theorem \ref{th1},  the minimization problem \eqref{e89}--\eqref{boundary} admits at least one minimizing solution pair
	\begin{align}(m,\overline{Q}_{e,s})\in
	\mathcal{A}=\big\{(m,\overline{Q}_{e,s})\in{\rm H}^1(\omega, \mathbb{R}^3)\times{\rm H}^1(\omega, {\rm SO}(3))\,\,\big|\,\,\,  m\big|_{ \gamma_d}=m^* \big\}\,.
\end{align}
\end{corollary}

\section{Existence of minimizers for the Cosserat shell model of order $O(h^3)$}\setcounter{equation}{0}

In this section we consider only terms up to order $O(h^3)$ in the expression of the energy density. Therefore, we obtain the following two-dimensional minimization problem   for the
deformation of the midsurface $m:\omega {\to}\mathbb{R}^3$ and the microrotation of the shell
$\overline{Q}_{e,s}:\omega {\to} \textrm{SO}(3)$ solving on $\omega\subset\mathbb{R}^2$: minimize with respect to $ (m,\overline{Q}_{e,s}) $ the following functional
\begin{equation}\label{eh3}
I(m,\overline{Q}_{e,s})\!=\!\! \int_{\omega}
W^{(h^3)}\big(  \mathcal{E}_{m,s} ,\,  \mathcal{K}_{e,s} \big)
 \,{\rm det}(\nabla y_0|n_0)       \,  {{\rm d}a} - \overline{\Pi}(m,\overline{Q}_{e,s})\,,
\end{equation}
where
	the shell energy density $W^{(h^3)}(\mathcal{E}_{m,s}, \mathcal{K}_{e,s})$ is given by
\begin{align}\label{h3energy} W^{(h^3)}(\mathcal{E}_{m,s}, \mathcal{K}_{e,s})=&\,  \Big(h+{\rm K}\,\dfrac{h^3}{12}\Big)\,
W_{\mathrm{shell}}\big(    \mathcal{E}_{m,s} \big)+  \dfrac{h^3}{12}\,
W_{\mathrm{shell}}  \big(   \mathcal{E}_{m,s} \, {\rm B}_{y_0} +   {\rm C}_{y_0} \mathcal{K}_{e,s} \big) \notag \\&
-\dfrac{h^3}{3} \mathrm{ H}\,\mathcal{W}_{\mathrm{shell}}  \big(  \mathcal{E}_{m,s} ,
\mathcal{E}_{m,s}{\rm B}_{y_0}+{\rm C}_{y_0}\, \mathcal{K}_{e,s} \big)+
\dfrac{h^3}{6}\, \mathcal{W}_{\mathrm{shell}}  \big(  \mathcal{E}_{m,s} ,
( \mathcal{E}_{m,s}{\rm B}_{y_0}+{\rm C}_{y_0}\, \mathcal{K}_{e,s}){\rm B}_{y_0} \big)\vspace{2.5mm}\\
&+  \Big(h-{\rm K}\,\dfrac{h^3}{12}\Big)\,
W_{\mathrm{curv}}\big(  \mathcal{K}_{e,s} \big)    +  \dfrac{h^3}{12}
W_{\mathrm{curv}}\big(  \mathcal{K}_{e,s}   {\rm B}_{y_0} \,  \big),  \notag
\end{align}
with all the other quantities having the same expressions and  interpretations as in the theory up to order $O(h^5)$.

In \cite{GhibaNeffPartI} we have presented a comparison with the the general 6-parameter shell model \cite{Eremeyev06}.
While in the previous approaches \cite{Eremeyev06,Pietraszkiewicz-book04,Pietraszkiewicz10,NeffBirsan13} the dependence of the coefficients upon the curved initial shell configuration is not specified, in  our shell model, the constitutive coefficients are deduced from the three-dimensional formulation, while the influence of the curved initial shell configuration appears explicitly  in the expression of the coefficients  of the energies for the reduced two-dimensional variational problem.
Another major difference between our model and the previously considered   general 6-parameter shell model is that, even in the case of a  simplified theory of order $O(h^3)$, additional mixed terms like the membrane--bending part $-\dfrac{h^3}{3} \mathrm{ H}\,\mathcal{W}_{\mathrm{shell}}  \big(  \mathcal{E}_{m,s} ,
\mathcal{E}_{m,s}{\rm B}_{y_0}+{\rm C}_{y_0}\, \mathcal{K}_{e,s} \big)$ and
$\dfrac{h^3}{6}\, \mathcal{W}_{\mathrm{shell}}  \big(  \mathcal{E}_{m,s} ,
( \mathcal{E}_{m,s}{\rm B}_{y_0}+{\rm C}_{y_0}\, \mathcal{K}_{e,s}){\rm B}_{y_0} \big)$, as well as  $W_{\mathrm{curv}}\big(  \mathcal{K}_{e,s}   {\rm B}_{y_0} \,  \big) $, are included, which are otherwise difficult to guess. Therefore, an existence proof for the new $O(h^3)$-model is of independent interest.
First, we will show
\begin{proposition}\label{coerh3}{\rm [Coercivity in the theory including terms up to order $O(h^3)$]} Assume that the constitutive coefficients are  such that $\mu>0, \,\mu_{\rm c}>0$, $2\,\lambda+\mu> 0$, $b_1>0$, $b_2>0$ and $b_3>0$ and let $c_2^+$  denotes the smallest eigenvalue  of
	$
	W_{\mathrm{curv}}(  S ),
	$
	and $c_1^+$ and $ C_1^+>0$ denote the smallest and the largest eigenvalues of the quadratic form $W_{\mathrm{shell}}(  S)$.
	If the thickness $h$ satisfies  one of the following conditions:
	\begin{enumerate}
		\item[i)] $h|\kappa_1|<\frac{1}{2},\qquad h|\kappa_2|<\frac{1}{2}$ \quad and \quad
		$h^2<	\left(\dfrac{47}{4}\right)^2(5-2 \sqrt{6})\frac{ {c_2^+}}{C_1^+}$;
		\item[ii)] $h|\kappa_1|<\frac{1}{a}$ ,\qquad $h|\kappa_2|<\frac{1}{a}$ \quad with\quad  $a>\max\Big\{1 + \frac{\sqrt{2}}{2},\frac{1+\sqrt{1+3\frac{C_1^+}{c_1^+}}}{2}\Big\}$,
	\end{enumerate}
	then $W^{(h^3)}(\mathcal{E}_{m,s}, \mathcal{K}_{e,s})$
	is coercive, in the sense that  there exists   a constant $a_1^+>0$ such that
	\begin{equation}\label{26bis}
	W^{(h^3)}(\mathcal{E}_{m,s}, \mathcal{K}_{e,s})\,\geq\, a_1^+\, \big(  \lVert \mathcal{E}_{m,s}\rVert ^2 + \lVert \mathcal{K}_{e,s}\rVert ^2\,\big) ,
	\end{equation}
	where
	$a_1^+$ depends on the constitutive coefficients.
\end{proposition}
\begin{proof}Using the properties presented in the Appendix,  since ${\rm B}_{y_0}^2-2\,{\rm H}\, {\rm B}_{y_0}+{\rm K}\, {\rm A}_{y_0}\,=\,0_3$ and $\mathcal{E}_{m,s}{\rm A}_{y_0}=\mathcal{E}_{m,s}$, it follows
	\begin{align}
	( \mathcal{E}_{m,s}{\rm B}_{y_0}+{\rm C}_{y_0}\, \mathcal{K}_{e,s}){\rm B}_{y_0}	=2\,{\rm H}\,
	\mathcal{E}_{m,s} {\rm B}_{y_0}-{\rm K}\, \mathcal{E}_{m,s}+{\rm C}_{y_0}\, \mathcal{K}_{e,s}{\rm B}_{y_0}.
	\end{align}
	Hence, we have
	\begin{align}\label{rescriere}
	\Big(h&+{\rm K}\,\dfrac{h^3}{12}\Big)\,
	W_{\mathrm{shell}}\big(    \mathcal{E}_{m,s} \big)
	-\dfrac{h^3}{3} \mathrm{ H}\,\mathcal{W}_{\mathrm{shell}}  \big(  \mathcal{E}_{m,s} ,
	\mathcal{E}_{m,s}{\rm B}_{y_0}+{\rm C}_{y_0}\, \mathcal{K}_{e,s} \big)+
	\dfrac{h^3}{6}\, \mathcal{W}_{\mathrm{shell}}  \big(  \mathcal{E}_{m,s} ,
	( \mathcal{E}_{m,s}{\rm B}_{y_0}+{\rm C}_{y_0}\, \mathcal{K}_{e,s}){\rm B}_{y_0} \big)\notag\\
	=& \Big(h+{\rm K}\,\dfrac{h^3}{12}\Big)\,
	W_{\mathrm{shell}}\big(    \mathcal{E}_{m,s} \big)
	-\dfrac{h^3}{3} \mathrm{ H}\,\mathcal{W}_{\mathrm{shell}}  \big(  \mathcal{E}_{m,s} ,
	\mathcal{E}_{m,s}{\rm B}_{y_0}\big)-\dfrac{h^3}{3} \mathrm{ H}\,\mathcal{W}_{\mathrm{shell}}  \big(  \mathcal{E}_{m,s} ,{\rm C}_{y_0}\, \mathcal{K}_{e,s} \big)\\&+
	\,{\rm H}\,\dfrac{h^3}{3}\, \mathcal{W}_{\mathrm{shell}}  \big(  \mathcal{E}_{m,s} ,
	\mathcal{E}_{m,s} {\rm B}_{y_0}\big)-{\rm K}\, \dfrac{h^3}{6}\, \mathcal{W}_{\mathrm{shell}}  \big(  \mathcal{E}_{m,s},\mathcal{E}_{m,s}\big)+\dfrac{h^3}{6}\, \mathcal{W}_{\mathrm{shell}}  \big(  \mathcal{E}_{m,s},{\rm C}_{y_0}\, \mathcal{K}_{e,s}{\rm B}_{y_0} \big)\notag\\
	=& \Big(h-{\rm K}\,\dfrac{h^3}{12}\Big)\,
	W_{\mathrm{shell}}\big(    \mathcal{E}_{m,s} \big)
	-\dfrac{h^3}{3} \mathrm{ H}\,\mathcal{W}_{\mathrm{shell}}  \big(  \mathcal{E}_{m,s} ,{\rm C}_{y_0}\, \mathcal{K}_{e,s} \big)+\dfrac{h^3}{6}\, \mathcal{W}_{\mathrm{shell}}  \big(  \mathcal{E}_{m,s},{\rm C}_{y_0}\, \mathcal{K}_{e,s}{\rm B}_{y_0} \big).\notag
	\end{align}
	Using \eqref{rescriere} and the positive definiteness of the quadratic forms \eqref{quadraticforms} and  the Cauchy--Schwarz  inequality we obtain
	\begin{align}
		W^{(h^3)}(\mathcal{E}_{m,s}, \mathcal{K}_{e,s})\geq \, &\Big(h-{\rm K}\,\dfrac{h^3}{12}\Big)
	W_{\mathrm{shell}}\big(    \mathcal{E}_{m,s} \big)+\dfrac{h^3}{12}\,
	W_{\mathrm{shell}}  \big(   \mathcal{E}_{m,s} \, {\rm B}_{y_0} +   {\rm C}_{y_0} \mathcal{K}_{e,s} \big)\notag\\& -
	\dfrac{h^3}{6}\,  |\mathcal{W}_{\mathrm{shell}}  \big(  \mathcal{E}_{m,s} ,\, {\rm C}_{y_0} \mathcal{K}_{e,s} {\rm B}_{y_0}\big)| -
	\dfrac{h^3}{3}\,|{\rm H}|\,  |\mathcal{W}_{\mathrm{shell}}  \big(  \mathcal{E}_{m,s} ,  \,   {\rm C}_{y_0} \mathcal{K}_{e,s} \big)|\notag\\&+\Big(h-{\rm K}\,\dfrac{h^3}{12}\Big)\,
	W_{\mathrm{curv}}\big(  \mathcal{K}_{e,s} \big)    +  \dfrac{h^3}{12}\,
	W_{\mathrm{curv}}\big(  \mathcal{K}_{e,s}   {\rm B}_{y_0} \,  \big)  \notag
	\\
	\geq \, &\Big(h-{\rm K}\,\dfrac{h^3}{12}\Big)\,
	W_{\mathrm{shell}}\big(    \mathcal{E}_{m,s} \big)-
	\dfrac{1}{6}\,  \left[h\,{W}_{\mathrm{shell}}  \big(  \mathcal{E}_{m,s}\big)\right]^{\frac{1}{2}}\, \left[h^5\,{W}_{\mathrm{shell}}  \big(
	{\rm C}_{y_0}\, \mathcal{K}_{e,s} {\rm B}_{y_0} \big)\right]^{\frac{1}{2}}\\&-
	\dfrac{1}{3}\,|{\rm H}|\,  \left[h^2\,{W}_{\mathrm{shell}}  \big(  \mathcal{E}_{m,s}\big)\right]^{\frac{1}{2}}\, \left[h^4\,{W}_{\mathrm{shell}}  \big(
	{\rm C}_{y_0}\, \mathcal{K}_{e,s}  \big)\right]^{\frac{1}{2}}+\dfrac{h^3}{12}\,
	W_{\mathrm{shell}}  \big(   \mathcal{E}_{m,s} \, {\rm B}_{y_0} +   {\rm C}_{y_0} \mathcal{K}_{e,s} \big)
	\notag\\&+\Big(h-{\rm K}\,\dfrac{h^3}{12}\Big)\,
	W_{\mathrm{curv}}\big(  \mathcal{K}_{e,s} \big) +  \dfrac{h^3}{12}\,
	W_{\mathrm{curv}}\big(  \mathcal{K}_{e,s}   {\rm B}_{y_0} \,  \big). \notag
	\end{align}
	In view of the arithmetic-geometric mean inequality it follows
	\begin{align}
	W^{(h^3)}(\mathcal{E}_{m,s}, \mathcal{K}_{e,s})	\geq \, &\Big(h-{\rm K}\,\dfrac{h^3}{12}\Big)\,
	W_{\mathrm{shell}}\big(    \mathcal{E}_{m,s} \big)-
	\dfrac{1}{12}\, \varepsilon \, h\,{W}_{\mathrm{shell}}  \big(  \mathcal{E}_{m,s}\big)-\dfrac{1}{12\,\varepsilon}\,  h^5\,{W}_{\mathrm{shell}}  \big(
	{\rm C}_{y_0}\, \mathcal{K}_{e,s} {\rm B}_{y_0} \big)\notag\\&-
	\dfrac{1}{6}\,\delta \,|{\rm H}|\,  h^2\,{W}_{\mathrm{shell}}  \big(  \mathcal{E}_{m,s}\big)-\dfrac{1}{6\,\delta} \,|{\rm H}|\,  h^4\,{W}_{\mathrm{shell}}  \big(
	{\rm C}_{y_0}\, \mathcal{K}_{e,s}  \big)+\dfrac{h^3}{12}\,
	W_{\mathrm{shell}}  \big(   \mathcal{E}_{m,s} \, {\rm B}_{y_0} +   {\rm C}_{y_0} \mathcal{K}_{e,s} \big)
	\notag\\&+\Big(h-{\rm K}\,\dfrac{h^3}{12}\Big)\,
	W_{\mathrm{curv}}\big(  \mathcal{K}_{e,s} \big) +  \dfrac{h^3}{12}\,
	W_{\mathrm{curv}}\big(  \mathcal{K}_{e,s}   {\rm B}_{y_0} \,  \big)
	\quad \forall\, \varepsilon>0\  \text{and} \ \delta>0.
	\end{align}

	Using the inequalities $-h^2 |K|>-\frac{1}{4}$ and $-h\, |H|>-\frac{1}{2}$, we obtain
	\begin{align}
		W^{(h^3)}(\mathcal{E}_{m,s}, \mathcal{K}_{e,s})\geq \, &
	\dfrac{h}{12}	\Big(\dfrac{47}{4}-\delta-
	\varepsilon\Big)\, {W}_{\mathrm{shell}}  \big(  \mathcal{E}_{m,s}\big)+\dfrac{h^3}{12}\, {W}_{\mathrm{shell}}  \big(
	\mathcal{E}_{m,s}{\rm B}_{y_0}+{\rm C}_{y_0}\, \mathcal{K}_{e,s}  \big)\notag\\&-\dfrac{1}{12\,\delta} \, h^3\,{W}_{\mathrm{shell}}  \big(
	{\rm C}_{y_0}\, \mathcal{K}_{e,s}  \big)-\dfrac{1}{12\,\varepsilon}\,  h^5\,{W}_{\mathrm{shell}}  \big(
	{\rm C}_{y_0}\, \mathcal{K}_{e,s} {\rm B}_{y_0} \big)\\&+\dfrac{47\,h}{48}\,
	W_{\mathrm{curv}}\big(  \mathcal{K}_{e,s} \big) +  \dfrac{h^3}{12}\,
	W_{\mathrm{curv}}\big(  \mathcal{K}_{e,s}   {\rm B}_{y_0} \,  \big) \quad \forall\, \varepsilon>0\  \text{and} \ \delta>0.\notag
	\end{align}
	In view of  \eqref{condition} and \eqref{pozitivdef} and 	since the Frobenius norm is sub-multiplicative, we deduce
	\begin{align}\label{dupac}
	W^{(h^3)}(\mathcal{E}_{m,s}, \mathcal{K}_{e,s})\geq\, &
	\dfrac{h}{12}	\Big(\dfrac{47}{4}-\delta-
	\varepsilon\Big)\,  c_1^+ \lVert \mathcal{E}_{m,s}\rVert ^2+\dfrac{h^3}{12}\, c_1^+ \lVert
	\mathcal{E}_{m,s}{\rm B}_{y_0}+{\rm C}_{y_0}\, \mathcal{K}_{e,s}  \rVert ^2\notag\\&-\dfrac{1}{12\,\delta} \, h^3\,C_1^+\lVert
	{\rm C}_{y_0}\rVert ^2\, \lVert \mathcal{K}_{e,s}\rVert ^2-\dfrac{1}{12\,\varepsilon }\,  h^5\,C_1^+\lVert
	{\rm C}_{y_0}\rVert ^2\, \lVert \mathcal{K}_{e,s}\,{\rm B}_{y_0} \rVert ^2\\&+\dfrac{47\,h}{48}\,c_2^+\lVert  \mathcal{K}_{e,s}\rVert ^2 +  \dfrac{h^3}{12}\,c_2^+\lVert    \mathcal{K}_{e,s}   {\rm B}_{y_0}\rVert ^2 \quad \forall\, \varepsilon>0\  \text{and} \ \delta>0 \ \text{such that}\  \dfrac{47}{4}>\delta+
	\varepsilon.\notag
	\end{align}

Since $
	\lVert C_{y_0}\rVert ^2=2
	$,  the estimate \eqref{dupac} becomes
	\begin{align}\label{oldineq}
	W^{(h^3)}(\mathcal{E}_{m,s}, \mathcal{K}_{e,s})\geq\,&
	\dfrac{h}{12} \Big(\dfrac{47}{4}-\delta-
	\varepsilon \Big)\,\,\,  c_1^+ \lVert \mathcal{E}_{m,s}\rVert ^2+\dfrac{h^3}{12}\, c_1^+ \lVert
	\mathcal{E}_{m,s}{\rm B}_{y_0}+{\rm C}_{y_0}\, \mathcal{K}_{e,s}  \rVert ^2 \notag\\&+ C_1^+\frac{ h^3}{6}\Big(\dfrac{47}{8}\,\frac{1}{h^2}\,\frac{c_2^+}{C_1^+}-\dfrac{1}{\delta} \,\Big)\lVert  \mathcal{K}_{e,s}\rVert ^2+ C_1^+ \frac{h^5}{6}\,\Big(\dfrac{1}{2}\,
	\frac{1}{h^2}\,\frac{c_2^+}{C_1^+}-\dfrac{1}{\varepsilon }\Big)\lVert    \mathcal{K}_{e,s}   {\rm B}_{y_0}\rVert ^2,
	\end{align}
	for all $\varepsilon>0\  \text{and} \ \delta>0$ \ \text{such that}\  $\dfrac{47}{4}>\delta+
	\varepsilon$.
	According to the properties presented in the Appendix, we have
	\begin{align}\label{B2}
	\lVert B_{y_0}\rVert ^2&=\bigl\langle  B_{y_0}^2, \id \bigr\rangle =2\, {\rm H}\,\bigl\langle   B_{y_0},\id  \bigr\rangle -{\rm K} \, \bigl\langle  A_{y_0}, \id \bigr\rangle =4\, {\rm H}^2-2\,{\rm K}.
	\end{align}
	Therefore, from \eqref{oldineq} it follows
	\begin{align}
W^{(h^3)}(\mathcal{E}_{m,s}, \mathcal{K}_{e,s})\geq\,  &
	\dfrac{h}{12} \Big(\dfrac{47}{4}-\delta-
	\varepsilon \Big)\,  c_1^+ \lVert \mathcal{E}_{m,s}\rVert ^2+\dfrac{h^3}{12}\, c_1^+ \lVert
	\mathcal{E}_{m,s}{\rm B}_{y_0}+{\rm C}_{y_0}\, \mathcal{K}_{e,s}  \rVert ^2\\&+h\,\Big[\dfrac{47}{48}\,c_2^+-\dfrac{1}{6\,\delta} \, h^2\,C_1^+\, -\dfrac{1}{6\,\varepsilon}\,  h^4\,C_1^+\,(4\, {\rm H}^2-2\,{\rm K})\Big]\lVert  \mathcal{K}_{e,s}\rVert ^2 +  h^3\dfrac{1}{12}\,c_2^+\lVert    \mathcal{K}_{e,s}   {\rm B}_{y_0}\rVert ^2.\notag
	\end{align}
	Using again that $h$ is small,
	we obtain $-h^2\,(4\, {\rm H}^2-2\,{\rm K})>-\frac{3}{2}$ and
	\begin{align}
W^{(h^3)}(\mathcal{E}_{m,s}, \mathcal{K}_{e,s})\geq\,&
	\dfrac{h}{12} \Big(\dfrac{47}{4}-\delta-
	\varepsilon \Big)\,  c_1^+ \lVert \mathcal{E}_{m,s}\rVert ^2+\dfrac{h^3}{12}\, c_1^+ \lVert
	\mathcal{E}_{m,s}{\rm B}_{y_0}+{\rm C}_{y_0}\, \mathcal{K}_{e,s}  \rVert ^2\\&+\frac{h^3}{12}\,C_1^+\,\Big[\dfrac{47}{4}\frac{c_2^+}{h^2\,C_1^+}-\dfrac{2}{\delta} \, \, -\dfrac{3}{\varepsilon}\,  \Big]\lVert  \mathcal{K}_{e,s}\rVert ^2 +  h^{3}\dfrac{1}{12}\,c_2^+\lVert    \mathcal{K}_{e,s}   {\rm B}_{y_0}\rVert ^2\notag\\
	\geq\,&
	\dfrac{h}{12} \Big(\dfrac{47}{4}-\delta-
	\varepsilon \Big)\,  c_1^+ \lVert \mathcal{E}_{m,s}\rVert ^2+\frac{h^3}{12}\,C_1^+\,\Big[\dfrac{47}{4}\frac{\,c_2^+}{h^2\,C_1^+}-\dfrac{2}{\delta} \, \, -\dfrac{3}{\varepsilon}\,  \Big]\lVert  \mathcal{K}_{e,s}\rVert ^2 .\notag
	\end{align}
	We consider $\delta=\gamma \, \varepsilon$ and we choose $\delta>0$ and $\gamma>0$ such that
	$
	\frac{47}{4\,(1+\gamma)}>\varepsilon>\frac{2+3\, \gamma}{\gamma}\frac{4}{47}\frac{h^2\,C_1^+}{c_2^+}.
	$
	This choice of the variable $\varepsilon$ is possible  if and
	only if
	$
	\left(\frac{47}{4}\right)^2\frac{\gamma}{(2+3\, \gamma)\,(1+\gamma)}>\frac{h^2\,C_1^+}{c_2^+}.
	$
	At this point we use that
	$
	\max_{\gamma>0}\frac{\gamma}{(2+3\, \gamma)\,(1+\gamma)}=5-2 \sqrt{6},
	$
	and we take $\gamma=\sqrt{\frac{2}{3}}$. Note that the considered values for $\gamma$ and $\varepsilon$ assure that the condition  $\dfrac{47}{4}>\delta+
	\varepsilon$ is automatically satisfied. Hence, we arrive at the following condition on the thickness $h$:
	\begin{align}
	h^2<\left(\frac{47}{4}\right)^2(5-2 \sqrt{6})\frac{ {c_2^+}}{C_1^+}\approx13.94\frac{ {c_2^+}}{C_1^+},
	\end{align}
	which proves the coercivity if the condition i) is satisfied.

	Next, we consider coercivity for condition ii). Under the hypotheses of the theorem, using also the positive definiteness of the quadratic forms \eqref{quadraticforms} and  the Cauchy--Schwarz  inequality,  we have
	\begin{align}
	W^{(h^3)}(\mathcal{E}_{m,s}, \mathcal{K}_{e,s})
	\geq\,&\Big(h+{\rm K}\,\dfrac{h^3}{12}\Big)\,
	W_{\mathrm{shell}}\big(    \mathcal{E}_{m,s} \big)-
	\dfrac{1}{3}\,|\mathrm{H}|\,  \left[h^2\,{W}_{\mathrm{shell}}  \big(  \mathcal{E}_{m,s}\big)\right]^{\frac{1}{2}}\, \left[h^4\,{W}_{\mathrm{shell}}  \big(
	\mathcal{E}_{m,s}{\rm B}_{y_0}+{\rm C}_{y_0}\, \mathcal{K}_{e,s}  \big)\right]^{\frac{1}{2}}\notag\\&+\dfrac{h^3}{12}\,
	W_{\mathrm{shell}}  \big(   \mathcal{E}_{m,s} \, {\rm B}_{y_0} +   {\rm C}_{y_0} \mathcal{K}_{e,s} \big)-\dfrac{1}{6}\,  \Big[h{W}_{\mathrm{shell}}  \big(  \mathcal{E}_{m,s}\big)]^{\frac{1}{2}}
	\, \Big[h^5{W}_{\mathrm{shell}}  \big((  \mathcal{E}_{m,s}{\rm B}_{y_0}+{\rm C}_{y_0}\, \mathcal{K}_{e,s}){\rm B}_{y_0} \big)]^{\frac{1}{2}}
	\notag\\&+\Big(h-{\rm K}\,\dfrac{h^3}{12}\Big)\,
	W_{\mathrm{curv}}\big(  \mathcal{K}_{e,s} \big).
	\end{align}
	Using the arithmetic-geometric mean inequality in the previous estimate, it follows
	\begin{align}
		W^{(h^3)}(\mathcal{E}_{m,s}, \mathcal{K}_{e,s})\geq\,&\Big(h+{\rm K}\,\dfrac{h^3}{12}-
	\dfrac{h^2}{6}\varepsilon\,|\mathrm{H}|\Big)\,\,\,  {W}_{\mathrm{shell}}  \big(  \mathcal{E}_{m,s}\big)+\Big(\dfrac{h^3}{12}\,-\dfrac{h^4}{6\, \varepsilon}\,\,|\mathrm{H}|\Big)\, {W}_{\mathrm{shell}}  \big(
	\mathcal{E}_{m,s}{\rm B}_{y_0}+{\rm C}_{y_0}\, \mathcal{K}_{e,s}  \big)\notag\\&-\dfrac{h}{12}\delta\,  {W}_{\mathrm{shell}}  \big(  \mathcal{E}_{m,s}\big)-
	\dfrac{h^5}{12\, \delta}{W}_{\mathrm{shell}}  \big(  \mathcal{E}_{m,s}{\rm B}_{y_0}+{\rm C}_{y_0}\, \mathcal{K}_{e,s}){\rm B}_{y_0} \big)\\&+\Big(h-{\rm K}\,\dfrac{h^3}{12}\Big)\,
	W_{\mathrm{curv}}\big(  \mathcal{K}_{e,s} \big)\notag.
	\end{align}
	We choose $\delta=8$ and $\varepsilon=2$ to obtain that
	\begin{align}
	W^{(h^3)}(\mathcal{E}_{m,s}, \mathcal{K}_{e,s})
	\geq\,&h\,\Big[\dfrac{1}{3}-{\rm K}\,\dfrac{h^2}{12}-
	\dfrac{h}{3}\,|\mathrm{H}|\Big]\,  {W}_{\mathrm{shell}}  \big(  \mathcal{E}_{m,s}\big)+\dfrac{h^3}{12}\Big(1\,-h\,|\mathrm{H}|\Big)\, {W}_{\mathrm{shell}}  \big(
	\mathcal{E}_{m,s}{\rm B}_{y_0}+{\rm C}_{y_0}\, \mathcal{K}_{e,s}  \big)\notag\\&-
	\dfrac{h^5}{96}\,\,
	W_{\mathrm{shell}} \big((  \mathcal{E}_{m,s} \, {\rm B}_{y_0} +  {\rm C}_{y_0} \mathcal{K}_{e,s} )   {\rm B}_{y_0} \,\big)+\Big(h-|{\rm K}|\,\dfrac{h^3}{12}\Big)\,
	W_{\mathrm{curv}}\big(  \mathcal{K}_{e,s} \big).
	\end{align}

	Let us consider $a>0$ and impose  $h\,|H|<\frac{1}{a}$,  $h^2\,|K|<\frac{1}{a^2}$. Therefore, using   \eqref{pozitivdef} and since the Frobenius norm is sub-multiplicative we deduce
	\begin{align}
W^{(h^3)}(\mathcal{E}_{m,s}, \mathcal{K}_{e,s})
	\geq\,&h\,\dfrac{4\, a^2-4\, a-1}{12\, a^2}\,  c_1^+ \lVert \mathcal{E}_{m,s}\rVert ^2+h\,\dfrac{12\, a^2-1}{12\,a^2}\,
	c_2^+  \lVert \mathcal{K}_{e,s} \rVert ^2\\
	&+\dfrac{h^3}{12}\dfrac{a-1}{a}\, c_1^+ \lVert
	\mathcal{E}_{m,s}{\rm B}_{y_0}+{\rm C}_{y_0}\, \mathcal{K}_{e,s} \rVert ^2-
	\dfrac{h^5}{96}\,\,
	C_1^+\lVert   \mathcal{E}_{m,s} \, {\rm B}_{y_0} +  {\rm C}_{y_0} \mathcal{K}_{e,s} \rVert ^2\,\lVert    {\rm B}_{y_0} \rVert ^2.\notag
	\end{align}
	Moreover, using \eqref{B2} we deduce $-h^2\lVert B_{y_0}\rVert^2>-\frac{6}{a^2}$ and the inequality
	\begin{align}\label{vreaM}
	W^{(h^3)}(\mathcal{E}_{m,s}, \mathcal{K}_{e,s})
	\geq\,&h\,\dfrac{4\, a^2-4\, a-1}{12\, a^2}\,  c_1^+ \lVert \mathcal{E}_{m,s}\rVert ^2+h\,\dfrac{12\, a^2-1}{12\,a^2}\,
	c_2^+  \lVert \mathcal{K}_{e,s} \rVert ^2\\
	&+\dfrac{h^3}{12}	C_1^+\dfrac{a-1}{a}\,\left[\frac{ c_1^+}{	C_1^+} -\frac{3}{4\,a (a-1)}\,\right]\lVert   \mathcal{E}_{m,s} \, {\rm B}_{y_0} +  {\rm C}_{y_0} \mathcal{K}_{e,s} \rVert ^2\,.\notag
	\end{align}
	Hence, choosing  $a>1 + \frac{\sqrt{2}}{2}$ we assure that
	$
	\frac{4\, a^2-4\, a-1}{12\, a^2}\, >0,$   $\frac{a-1}{a}>0$  \text{and}   $\frac{12\, a^2-1}{12\,a^2}>0.
	$
	A suitable  $a>1 + \frac{\sqrt{2}}{2}$ should  satisfy
	$
	\frac{ c_1^+}{	C_1^+} -\frac{3}{4\,a (a-1)}>0,
	$
	which is true if $a$ is such that
	$
	a>\frac{1+\sqrt{1+3\frac{C_1^+}{c_1^+}}}{2}>1.
	$
	Therefore, if $a>\max\Big\{1 + \frac{\sqrt{2}}{2},\frac{1+\sqrt{1+3\frac{C_1^+}{c_1^+}}}{2}\Big\}$,
	then  inequality \eqref{vreaM}  yields
	\begin{align}
	W^{(h^3)}(\mathcal{E}_{m,s}, \mathcal{K}_{e,s})
	\geq\,&h\,\dfrac{4\, a^2-4\, a-1}{12\, a^2}\,  c_1^+ \lVert \mathcal{E}_{m,s}\rVert ^2+h\,\dfrac{12\, a^2-1}{12\,a^2}\,
	c_2^+  \lVert \mathcal{K}_{e,s} \rVert ^2.\notag\qedhere
	\end{align}

\end{proof}

\begin{corollary}{\rm [Uniform convexity for the theory including terms up to order $O(h^3)$]}\label{corconh3} Under the hypotheses of Proposition \ref{coerh3},   the  energy density
$	W^{(h^3)}(\mathcal{E}_{m,s}, \mathcal{K}_{e,s})$
	is uniformly convex in $(\mathcal{E}_{m,s}, \mathcal{K}_{e,s})$, i.e.,  there exists a constant $a_1^+>0$ such that
	\begin{align}
	D^2\,	W^{(h^3)}(\mathcal{E}_{m,s}, \mathcal{K}_{e,s}).\,[(H_1,H_2),(H_1,H_2)]\geq a_1^+(\lVert H_1\rVert^2+\lVert H_2\rVert^2) \qquad \forall\, H_1,H_2\in \mathbb{R}^{3\times 3}.
	\end{align}
\end{corollary}
\begin{proof}
	See Corollary \ref{corconh5}.
\end{proof}
Therefore, an existence result similar to Theorem \ref{th1} holds true for the theory including terms up to order $O (h^3)$:
\begin{theorem}\label{th11}{\rm [Existence result for the theory including terms up to order $O(h^3)$]}
	Assume that the external loads satisfy the conditions
	\begin{equation}\label{24a}
	{f}\in\textrm{\rm L}^2(\omega,\mathbb{R}^3),\qquad  t\in \textrm{\rm L}^2(\gamma_t,\mathbb{R}^3),
	\end{equation}
 the boundary data satisfy the conditions
	\begin{equation}\label{25a}
	{m}^*\in{\rm H}^1(\omega ,\mathbb{R}^3),\qquad \overline{Q}_{e,s}^*\in{\rm H}^1(\omega, {\rm SO}(3)),
	\end{equation}
and that the following conditions concerning the initial configuration are fulfilled: $\,y_0:\omega\subset\mathbb{R}^2\rightarrow\mathbb{R}^3$ is a continuous injective mapping and
	\begin{align}\label{26a}
	{y}_0&\in{\rm H}^1(\omega ,\mathbb{R}^3),\qquad   {Q}_{0}(0)\in{\rm H}^1(\omega, {\rm SO}(3)),\qquad
	\nabla_x\Theta(0)\in {\rm L}^\infty(\omega ,\mathbb{R}^{3\times 3}),\qquad \det[\nabla_x\Theta(0)] \geq\, a_0 >0\,,
	\end{align}
	where $a_0$ is a constant.
	 Assume that the constitutive coefficients are  such that $\mu>0, \,\mu_{\rm c}>0$, $2\,\lambda+\mu> 0$, $b_1>0$, $b_2>0$ and $b_3>0$.
	Then, if the thickness $h$ satisfies at least one of the following conditions:
	\begin{enumerate}
		\item[i)] $h|\kappa_1|<\frac{1}{2},\qquad h|\kappa_2|<\frac{1}{2}$ \quad and \quad
		$h^2<	\left(\dfrac{47}{4}\right)^2(5-2 \sqrt{6})\frac{ {c_2^+}}{C_1^+}$;
		\item[ii)] $h|\kappa_1|<\frac{1}{a}$ ,\qquad $h|\kappa_2|<\frac{1}{a}$ \quad with\quad  $a>\max\Big\{1 + \frac{\sqrt{2}}{2},\frac{1+\sqrt{1+3\frac{C_1^+}{c_1^+}}}{2}\Big\}$,
	\end{enumerate}
where  $c_2^+$  denotes the smallest eigenvalue  of
$
W_{\mathrm{curv}}(  S ),
$
and $c_1^+$ and $ C_1^+>0$ denote the smallest and the biggest eigenvalues of the quadratic form $W_{\mathrm{shell}}(  S)$,
	 the minimization problem corresponding to the energy density defined by \eqref{eh3} and \eqref{h3energy}  admits at least one minimizing solution pair
	$(m,\overline{Q}_{e,s})\in \mathcal{A}$.
\end{theorem}

\begin{corollary}\label{th1nD}{\rm [Existence result for the theory including terms up to order $O(h^3)$ without boundary condition on the microrotation field]}
Under the hypotheses of Theorem \ref{th11}, the minimization problem corresponding to the energy density defined by \eqref{eh3} and \eqref{h3energy}  admits at least one minimizing solution pair
	\begin{align}(m,\overline{Q}_{e,s})\in
	\mathcal{A}=\big\{(m,\overline{Q}_{e,s})\in{\rm H}^1(\omega, \mathbb{R}^3)\times{\rm H}^1(\omega, {\rm SO}(3))\,\,\big|\,\,\,  m\big|_{ \gamma_d}=m^* \big\}\,.
	\end{align}
\end{corollary}
\section{Final comments}\label{Conclusion}\setcounter{equation}{0}
Having the
deformation of the midsurface $m:\omega\subset\mathbb{R}^2\to\mathbb{R}^3$ and the microrotation of the shell
$\overline{Q}_{e,s}:\omega\subset\mathbb{R}^2\to \textrm{SO}(3)$ solving on $\omega$ the minimization (two-dimensional) problem, we get the approximation of the deformation of the initial  three-dimensional body using  the following \textit{6-parameter quadratic ansatz} in the thickness direction for the reconstructed total deformation $\varphi_s:\Omega_h\subset \mathbb{R}^3\rightarrow \mathbb{R}^3$ of the shell-like structure \cite{GhibaNeffPartI}
\begin{align}\label{ansatz}
\varphi_s(x_1,x_2,x_3)\,=\,&m(x_1,x_2)+\bigg(x_3\varrho_m(x_1,x_2)+\dd\frac{x_3^2}{2}\varrho_b(x_1,x_2)\bigg)\overline{Q}_{e,s}(x_1,x_2)\nabla_x\Theta(x_1,x_2,0)\, e_3\, ,
\end{align}
  where $\varrho_m,\,\varrho_b:\omega\subset\mathbb{R}^2\to \mathbb{R}$ allow in principal for symmetric thickness stretch  ($\varrho_m\neq1$) and asymmetric thickness stretch ($\varrho_b\neq 0$) about the midsurface and they are given by
\begin{align}\label{final_rho}
\varrho_m\,=\,&1-\frac{\lambda}{\lambda+2\,\mu}[\bigl\langle  \overline{Q}_{e,s}^T(\nabla m|0)[\nabla_x\Theta(0)]^{-1},\id_3 \bigr\rangle -2]
\,=\,1-\frac{\lambda}{\lambda+2\mu}\tr( \mathcal{E}_{m,s} ) \;,\notag \\
\dd\varrho_b\,=\,&-\frac{\lambda}{\lambda+2\,\mu}\bigl\langle  \overline{Q}_{e,s}^T(\nabla (\,\overline{Q}_{e,s}\nabla_x\Theta(0)\, e_3)|0)[\nabla_x\Theta(0)]^{-1},\id_3 \bigr\rangle   \\
&+\frac{\lambda}{\lambda+2\,\mu}\bigl\langle  \overline{Q}_{e,s}^T(\nabla m|0)[\nabla_x\Theta(0)]^{-1}(\nabla n_0|0)[\nabla_x\Theta(0)]^{-1},\id_3 \bigr\rangle
\,=\,-\frac{\lambda}{\lambda+2\mu}{\rm tr} [{\rm C}_{y_0}\mathcal{K}_{e,s} +\mathcal{E}_{m,s} {\rm B}_{y_0}] .\notag
\end{align}

Obviously,  if we know the total microrotation $\overline{R}_s(x_1,x_2) =\overline{Q}_{e,s}(x_1,x_2)\,Q_0(x_1,x_2,0)\in{\rm SO}(3)$, then we know  the microrotation $\overline{R}_\xi$ of the parental three-dimensional problem, since we assume it is independent of $x_3$.

It is noteworthy that the existence result in the $O(h^3)$-model is not simply the truncated version of the existence result for the $O(h^5)$-model. Both existence results require uniformly positive constitutive parameters, in particular    $\mu_{\rm c}>0$.  {However, the existence result for the $O(h^5)$-model needs less stringent
assumptions on the plate thickness versus the initial curvature. This may indicate
that the $O(h^5)$-model is intrinsically more stable than the $O(h^3)$-model. We will
illuminate this possible feature in future computational researches.
Since classical shell models do not have any drill contribution, it is interesting
to consider the no drill case generated by $\mu_{\rm c}=0$ (free drill, but no energy
contribution associated to this deformation mode).} In this interesting limit
case, we would need new generalized Korn's inequalities \cite{Neff00b,Pompe03,Pompe10,neff2014counterexamples}, which couple the smoothness of the rotation field $\overline{R}_s$ with the coercivity with respect to the deformation $m$, in the sense that
\begin{align}
\lVert \overline{R}_s^T \,(\nabla m|0)+(\nabla m|0)^T\overline{R}_s\rVert^2_{{\rm L}^2(\omega)}\geq c^+\lVert m\rVert^2_{{\rm H}^1(\omega)}\,.
\end{align}
However,  such an estimate is currently only known for $ \overline{R}_s\in C(\overline{\omega},{\rm SO}(3))$, but the elastic shell energy only assures $ \overline{R}_s\in {\rm H}^1(\omega,{\rm SO}(3))$. More research is needed in this direction.

\bigskip

\begin{footnotesize}
	\noindent{\bf Acknowledgements:}   This research has been funded by the Deutsche Forschungsgemeinschaft (DFG, German Research Foundation) -- Project no. 415894848 (M. B\^{\i}rsan, P. Lewintan and P. Neff). The  work of I.D. Ghiba  was supported by a grant of the Romanian Ministry of Research
	and Innovation, CNCS--UEFISCDI, project number
	PN-III-P1-1.1-TE-2019-0397, within PNCDI III.
\end{footnotesize}

\begin{footnotesize}
\bibliographystyle{plain} %plain

\addcontentsline{toc}{section}{References}

\begin{thebibliography}{10}
	
	\bibitem{Adams75}
	R.A. Adams.
	\newblock {\em Sobolev {S}paces.}, volume~65 of {\em Pure and {A}pplied
		{M}athematics}.
	\newblock Academic Press, London, 1. edition, 1975.
	
	\bibitem{Badur-Pietrasz86}
	J.~Badur and W.~Pietraszkiewicz.
	\newblock On geometrically non-linear theory of elastic shells derived from
	pseudo-{C}osserat continuum with constrained micro-rotations.
	\newblock In W.~Pietraszkiewicz, editor, {\em Finite Rotations in Structural
		Mechanics.}, pages 19--32. Springer, Berlin, 1986.
	
	\bibitem{Ball77}
	J.M. Ball.
	\newblock Convexity conditions and existence theorems in nonlinear elasticity.
	\newblock {\em Arch. Rational Mech. Anal.}, 63:337--403, 1977.
	
	\bibitem{Birsan08}
	M.~B\^{\i}rsan.
	\newblock Inequalities of {K}orn's type and existence results in the theory of
	{C}osserat elastic shells.
	\newblock {\em J. Elasticity}, 90:227--239, 2008.
	
	\bibitem{birsan2020derivation}
	M.~B{\^\i}rsan.
	\newblock Derivation of a refined 6-parameter shell model: {D}escent from the
	three-dimensional {C}osserat elasticity using a method of classical shell
	theory.
	\newblock {\em Math. Mech. Solids, doi.org/10.1177/1081286519900531}, 2020.
	
	\bibitem{birsan2019refined}
	M.~B\^irsan, I.D. Ghiba, R.J. Martin, and P.~Neff.
	\newblock Refined dimensional reduction for isotropic elastic {Cosserat} shells
	with initial curvature.
	\newblock {\em Math. Mech. Solids}, 24(12):4000--4019, 2019.
	
	\bibitem{NeffBirsan13}
	M.~B\^{i}rsan and P.~Neff.
	\newblock Existence of minimizers in the geometrically non-linear 6-parameter
	resultant shell theory with drilling rotations.
	\newblock {\em Math. Mech. Solids}, 19(4):376--397, 2014.
	
	\bibitem{Birsan-Neff-L54-2014}
	M.~B\^{\i}rsan and P.~Neff.
	\newblock Shells without drilling rotations: {A} representation theorem in the
	framework of the geometrically nonlinear 6-parameter resultant shell theory.
	\newblock {\em Int. J. Engng. Sci.}, 80:32--42, 2014.
	
	\bibitem{bunoiu2015existence}
	R.~Bunoiu, Ph.G. Ciarlet, and C.~Mardare.
	\newblock Existence theorem for a nonlinear elliptic shell model.
	\newblock {\em J. Elliptic Parabol. Equ.}, 1(1):31--48, 2015.
	
	\bibitem{Pietraszkiewicz-book04}
	J.~Chr\'o\'scielewski, J.~Makowski, and W.~Pietraszkiewicz.
	\newblock {\em Statics and Dynamics of Multifold Shells: Nonlinear Theory and
		Finite Element Method (in Polish).}
	\newblock Wydawnictwo IPPT PAN, Warsaw, 2004.
	
	\bibitem{Pietraszkiewicz10}
	J.~Chr\'o\'scielewski, W.~Pietraszkiewicz, and W.~Witkowski.
	\newblock On shear correction factors in the non-linear theory of elastic
	shells.
	\newblock {\em Int. J. Solids Struct.}, 47:3537--3545, 2010.
	
	\bibitem{Ciarlet97}
	Ph.G. Ciarlet.
	\newblock {\em Mathematical {E}lasticity, {V}ol. {II}: {T}heory of {P}lates.}
	\newblock North-Holland, Amsterdam, first edition, 1997.
	
	\bibitem{Ciarlet98}
	Ph.G. Ciarlet.
	\newblock {\em Introduction to Linear Shell Theory.}
	\newblock Gauthier-Villars, Paris, 1998.
	
	\bibitem{Ciarlet00}
	Ph.G. Ciarlet.
	\newblock {\em Mathematical {E}lasticity, {V}ol. {III}: {T}heory of {S}hells.}
	\newblock North-Holland, Amsterdam, first edition, 2000.
	
	\bibitem{ciarlet1982lois}
	Ph.G. Ciarlet and G.~Geymonat.
	\newblock Sur les lois de comportement en {\'e}lasticit{\'e} non lin{\'e}aire
	compressible.
	\newblock {\em C.R. Acad. Sci. Paris, Ser. II}, 295:423--426, 1982.
	
	\bibitem{ciarlet2013orientation}
	Ph.G. Ciarlet, R.~Gogu, and C.~Mardare.
	\newblock Orientation-preserving condition and polyconvexity on a surface:
	{Application} to nonlinear shell theory.
	\newblock {\em J. Math. Pures Appl.}, 99:704--725, 2013.
	
	\bibitem{ciarlet2018existence}
	Ph.G. Ciarlet and C.~Mardare.
	\newblock An existence theorem for a two-dimensional nonlinear shell model of
	{Koiter's} type.
	\newblock {\em Math. Models Methods Appl. Sci}, 28(14):2833--2861, 2018.
	
	\bibitem{Pietraszkiewicz04}
	V.A. Eremeyev and W.~Pietraszkiewicz.
	\newblock The nonlinear theory of elastic shells with phase transitions.
	\newblock {\em J. Elasticity}, 74:67--86, 2004.
	
	\bibitem{Eremeyev06}
	V.A. Eremeyev and W.~Pietraszkiewicz.
	\newblock Local symmetry group in the general theory of elastic shells.
	\newblock {\em J. Elasticity}, 85:125--152, 2006.
	
	\bibitem{Simo92}
	D.D. Fox and J.C. Simo.
	\newblock A drill rotation formulation for geometrically exact shells.
	\newblock {\em Comp. Meth. Appl. Mech. Eng.}, 98:329--343, 1992.
	
	\bibitem{GhibaNeffPartI}
	I.D. Ghiba, M.~B\^irsan, P.~Lewintan, and P.~Neff.
	\newblock The isotropic {C}osserat shell model including terms up to
	${O}(h^5)$. {Part I: D}erivation in matrix notation.
	\newblock {\em submitted, arXiv:2003.00549}.
	
	\bibitem{Raviart79}
	V.~Girault and P.A. Raviart.
	\newblock {\em Finite Element Approximation of the {N}avier-{S}tokes
		Equations.}, volume 749 of {\em Lect. Notes Math.}
	\newblock Springer, Heidelberg, 1979.
	
	\bibitem{Ibrahim94}
	A.~Ibrahimbegovi\'c.
	\newblock Stress resultant geometrically nonlinear shell theory with drilling
	rotations - {P}art {I}: {A} consistent formulation.
	\newblock {\em Comput. Meth. Appl. Mech. Eng.}, 118:265--284, 1994.
	
	\bibitem{Koiter60}
	W.T. Koiter.
	\newblock A consistent first approximation in the general theory of thin
	elastic shells.
	\newblock In W.T. Koiter, editor, {\em The Theory of Thin Elastic Shells},
	IUTAM Symposium Delft 1960, pages 12--33. North-Holland, Amsterdam, 1960.
	
	\bibitem{Koiter69}
	W.T. Koiter.
	\newblock Foundations and basic equations of shell theory. {A} survey of recent
	progress.
	\newblock In F.I. Niordson, editor, {\em Theory of Thin Shells}, IUTAM
	Symposium Copenhagen 1967, pages 93--105. Springer, Heidelberg, 1969.
	
	\bibitem{Leis86}
	R.~Leis.
	\newblock {\em {Initial Boundary Value Problems in Mathematical Physics}}.
	\newblock Teubner, Stuttgart, 1986.
	
	\bibitem{Neff00b}
	P.~Neff.
	\newblock On {K}orn's first inequality with nonconstant coefficients.
	\newblock {\em Proc. Roy. Soc. Edinb. A}, 132:221--243, 2002.
	
	\bibitem{Neff_plate04_cmt}
	P.~Neff.
	\newblock A geometrically exact {C}osserat-shell model including size effects,
	avoiding degeneracy in the thin shell limit. {P}art {I}: {F}ormal dimensional
	reduction for elastic plates and existence of minimizers for positive
	{C}osserat couple modulus.
	\newblock {\em Cont. Mech. Thermodynamics}, 16:577--628, 2004.
	
	\bibitem{Neff_Habil04}
	P.~Neff.
	\newblock {\em Geometrically exact {C}osserat theory for bulk behaviour and
		thin structures. {M}odelling and mathematical analysis.}
	\newblock Signatur HS 7/0973. Habilitationsschrift, Universit\"ats- und
	Landesbibliothek, Technische Universit\"at Darmstadt, Darmstadt, 2004.
	
	\bibitem{Neff_membrane_plate03}
	P.~Neff.
	\newblock A geometrically exact viscoplastic membrane-shell with viscoelastic
	transverse shear resistance avoiding degeneracy in the thin-shell limit.
	{P}art {I}: {T}he viscoelastic membrane-plate.
	\newblock {\em Z. Angew. Math. Phys.}, 56(1):148--182, 2005.
	
	\bibitem{Neff_membrane_existence03}
	P.~Neff.
	\newblock Local existence and uniqueness for a geometrically exact
	membrane-plate with viscoelastic transverse shear resistance.
	\newblock {\em Math. Methods Appl. Sci.}, 28:1031--1060, 2005.
	
	\bibitem{Neff_Danzig05}
	P.~Neff.
	\newblock The {$\Gamma$}-limit of a finite strain {C}osserat model for
	asymptotically thin domains versus a formal dimensional reduction.
	\newblock In W.~Pietraszkiewiecz and C.~Szymczak, editors, {\em
		Shell-Structures: Theory and Applications.}, pages 149--152. Taylor and
	Francis Group, London, 2006.
	
	\bibitem{Neff_plate07_m3as}
	P.~Neff.
	\newblock A geometrically exact planar {C}osserat shell-model with
	microstructure: Existence of minimizers for zero {C}osserat couple modulus.
	\newblock {\em Math. Mod. Meth. Appl. Sci.}, 17:363--392, 2007.
	
	\bibitem{neff2014existence}
	P.~Neff, M.~B\^{i}rsan, and F.~Osterbrink.
	\newblock Existence theorem for the classical nonlinear {C}osserat elastic
	model.
	\newblock {\em J. Elasticity}, 121(1):119--1, 2015.
	
	\bibitem{Neff_Chelminski_ifb07}
	P.~Neff and K.~Che{\l}mi\'nski.
	\newblock A geometrically exact {C}osserat shell-model for defective elastic
	crystals. {J}ustification via {$\Gamma$}-convergence.
	\newblock {\em Interfaces and Free Boundaries}, 9:455--492, 2007.
	
	\bibitem{Neff_Hong_Reissner08}
	P.~Neff, K.-I. Hong, and J.~Jeong.
	\newblock The {R}eissner-{M}indlin plate is the {$\Gamma$}-limit of {C}osserat
	elasticity.
	\newblock {\em Math. Mod. Meth. Appl. Sci.}, 20:1553--1590, 2010.
	
	\bibitem{neff2013grioli}
	P.~Neff, J.~Lankeit, and A.~Madeo.
	\newblock On {G}rioli's minimum property and its relation to {C}auchy's polar
	decomposition.
	\newblock {\em Int. J. Engng. Sci}, 80:207--217, 2014.
	
	\bibitem{Neff_curl06}
	P.~Neff and I.~M\"unch.
	\newblock Curl bounds {Grad} on {${\rm SO}(3)$}.
	\newblock {\em ESAIM: Control, Optimisation and Calculus of Variations},
	14(1):148--159, 2008.
	
	\bibitem{neff2014counterexamples}
	P.~Neff and W.~Pompe.
	\newblock Counterexamples in the theory of coerciveness for linear elliptic
	systems related to generalizations of {K}orn's second inequality.
	\newblock {\em Z. Angew. Math. Mech.}, 94:784--790, 2014.
	
	\bibitem{Paroni06}
	R.~Paroni, P.~Podio-Guidugli, and G.~Tomassetti.
	\newblock The {R}eissner-{M}indlin plate theory via {$\Gamma$}-convergence.
	\newblock {\em C. R. Acad. Sci. Paris, Ser. I}, 343:437--440, 2006.
	
	\bibitem{Pietraszkiewicz14}
	W.~Pietraszkiewicz and V.~Konopi\'nska.
	\newblock Drilling couples and refined constitutive equations in the resultant
	geometrically non-linear theory of elastic shells.
	\newblock {\em Int. J. Solids Struct.}, 51:2133--2143, 2014.
	
	\bibitem{Pompe03}
	W.~Pompe.
	\newblock Korn's first inequality with variable coefficients and its
	generalizations.
	\newblock {\em Comment. Math. Univ. Carolinae}, 44,1:57--70, 2003.
	
	\bibitem{Pompe10}
	W.~Pompe.
	\newblock Counterexamples to {K}orn's inequality with non-constant rotation
	coefficients.
	\newblock {\em Math. Mech. Solids}, 16:172--176, doi: 10.1177/1081286510367554,
	2011.
	
	\bibitem{Sansour92}
	C.~Sansour and H.~Bufler.
	\newblock An exact finite rotation shell theory, its mixed variational
	formulation and its finite element implementation.
	\newblock {\em Int. J. Num. Meth. Engng.}, 34:73--115, 1992.
	
	\bibitem{Simo89.1}
	J.C. Simo and D.D. Fox.
	\newblock On a stress resultant geometrically exact shell model. {P}art {I}:
	{F}ormulation and optimal parametrization.
	\newblock {\em Comp. Meth. Appl. Mech. Eng.}, 72:267--304, 1989.
	
	\bibitem{Tiba02}
	J.~Sprekels and D.~Tiba.
	\newblock An analytic approach to a generalized {N}aghdi shell model.
	\newblock {\em Adv. Math. Sci. Appl.}, 12:175--190, 2002.
	
	\bibitem{Steigmann08}
	D.J. Steigmann.
	\newblock Two-dimensional models for the combined bending and stretching of
	plates and shells based on three-dimensional linear elasticity.
	\newblock {\em Int. J. Engng. Sci}, 46:654--676, 2008.
	
	\bibitem{Steigmann12}
	D.J. Steigmann.
	\newblock Extension of {K}oiter's linear shell theory to materials exhibiting
	arbitrary symmetry.
	\newblock {\em Int. J. Engng. Sci.}, 51:216--232, 2012.
	
	\bibitem{tambavca2010semicontinuity}
	J.~Tamba{\v{c}}a and I.~Vel{\v{c}}i{\'c}.
	\newblock Semicontinuity theorem in the micropolar elasticity.
	\newblock {\em ESAIM: Control, Optimisation and Calculus of Variations},
	16(2):337--355, 2010.
	
	\bibitem{Tambaca10}
	J.~Tamba\v{c}a and I.~Vel\v{c}i\'{c}.
	\newblock Existence theorem for nonlinear micropolar elasticity.
	\newblock {\em ESAIM: Control, Optimisation and Calculus of Variations.},
	16:92--110, 2010.
	
	\bibitem{Neff_membrane_Weinberg07}
	K.~Weinberg and P.~Neff.
	\newblock A geometrically exact thin membrane model-investigation of large
	deformations and wrinkling.
	\newblock {\em Int. J. Numer. Methods Engrg.}, 74(6):871--893, 2008.
	
	\bibitem{Zhilin06}
	P.A. Zhilin.
	\newblock {\em Applied Mechanics -- Foundations of Shell Theory (in Russian)}.
	\newblock State Polytechnical University Publisher, Sankt Petersburg, 2006.
	
\end{thebibliography}

\appendix
\setcounter{equation}{0}

	\section*{Appendix. Properties of the considered tensors}\label{propABAppendix}\setcounter{section}{1}\setcounter{equation}{0}
\addcontentsline{toc}{section}{Appendix}
In this paper we use some properties of the tensors involved in the variational formulation of the shell model \cite{GhibaNeffPartI}.

\begin{proposition}\label{propAB}
	The following identities are satisfied :
	\begin{itemize}
		\item [i)]

		$\tr[{\rm A}_{y_0}]\,=\,2,$ \quad	${\det}[{\rm A}_{y_0}]\,=\,0;\quad $
		$\tr[{\rm B}_{y_0}]\,=\,2\,{\rm H}\,$,\quad  ${\det}[{\rm B}_{y_0}]\,=\,0,$

		\item[ii)] ${\rm B}_{y_0}$ satisfies the equation of Cayley-Hamilton type
		$
		{\rm B}_{y_0}^2-2\,{\rm H}\, {\rm B}_{y_0}+{\rm K}\, {\rm A}_{y_0}\,=\,0_3;
	$
		\item[iii)] ${\rm A}_{y_0}{\rm B}_{y_0}\,=\,{\rm B}_{y_0}{\rm A}_{y_0}\,=\,{\rm B}_{y_0}$, \quad  ${\rm A}_{y_0}^2\,=\,{\rm A}_{y_0}$;
		\item[iv)] ${\rm C}_{y_0}\in \mathfrak{so}(3)$, $\quad {\rm C}_{y_0}^2\,=\,-{\rm A}_{y_0}$ and it has the simplified form $
		{\rm C}_{y_0}:=Q_0(0)\begin{footnotesize}\begin{pmatrix}
		0&1&0 \\
		-1&0&0 \\
		0&0&0
		\end{pmatrix}\end{footnotesize}Q_0^T(0)\in \mathfrak{so}(3).
	$

		\item[v)] $
		\overline{Q}_{e,s}^T\,(\nabla [\overline{Q}_{e,s}\nabla_x\Theta (0)\, e_3]\,|\,0)\,[\nabla_x \Theta(0)]^{-1}\,\,=\,\,{\rm C}_{y_0} \mathcal{K}_{e,s}-{\rm B}_{y_0};
		$
		\item[vi)] ${\rm C}_{y_0} \mathcal{K}_{e,s} {\rm A}_{y_0}\,\,=\,\,{\rm C}_{y_0} \mathcal{K}_{e,s} $;
		\item[vii)] $\mathcal{E}_{m,s} {\rm A}_{y_0}\,\,=\,\,\mathcal{E}_{m,s} $.
	\end{itemize}
\end{proposition}
\begin{proof}
	For the proof of this proposition we refer to \cite{GhibaNeffPartI}. Here, we prove only the third identity of iv).

	We  have
	$
	[	\nabla_x \Theta(x_3)]\, e_3\,= \, n_0\,.
	$
	Let us recall that $X\in \mathrm{GL}^+(3)$ satisfies the \textit{\textbf{G}eneralized \textbf{K}irchhoff \textbf{C}onstraint} ({\rm GKC}) \cite{Neff_Habil04} if
	$
	X\in {\rm GKC }:=\{X\in \mathrm{GL}^+(3)\ |\ X^T \, X\, e_3\,=\,\varrho^2 e_3,\ \varrho\in \mathbb{R}^+\}\, .
	$
	For all $X\in {\rm GKC }$ with the polar decomposition $X=R\, U _0$, if follows that  $ U _0\in  {\rm GKC }$.
	In view of this property and $\nabla \Theta(x_3)=Q_0(x_3)U_0(x_3)$, it follows\footnote{Here, $*$ denotes quantities having expressions which are not relevant for our calculations.}
	$
	U _0(x_3)\,=\,\begin{footnotesize}
	\begin{pmatrix}
	* &* &0  \\
	* &* &0 \\
	0 &0 &1
	\end{pmatrix}
	\end{footnotesize}\in {\rm Sym}^+(3)\,.
	$
	Since  $\det Q_0\,=\,1$, we deduce
	\begin{align}\label{formaC}
	{\rm C}_{y_0}&\,=\,{\rm Cof}	(\nabla_x \Theta(0))\,\begin{footnotesize}\begin{pmatrix}
	0&1&0 \\
	-1&0&0 \\
	0&0&0
	\end{pmatrix}\end{footnotesize}\,  [	\nabla_x \Theta(0)]^{-1} \,=\,Q_0(0)\,(\det U_0(0))\,  U_0^{-1}(0)\,\begin{footnotesize}\begin{pmatrix}
	0&1&0 \\
	-1&0&0 \\
	0&0&0
	\end{pmatrix}\end{footnotesize}\, U_0^{-1}(0)\, Q_0^T(0).
	\end{align}
	Direct computations give us $\left(
		\begin{array}{ccc}
		a & x & 0 \\
		x & b & 0 \\
		0 & 0 & 1 \\
		\end{array}
		\right)\,\left(
		\begin{array}{ccc}
		0 & 1 & 0 \\
		-1 & 0 & 0 \\
		0 & 0 & 0 \\
		\end{array}
		\right)\,\left(
		\begin{array}{ccc}
		a & x & 0 \\
		x & b & 0 \\
		0 & 0 & 1 \\
		\end{array}
		\right)=\left(
		\begin{array}{ccc}
		0 & a b-x^2 & 0 \\
		x^2-a b & 0 & 0 \\
		0 & 0 & 0 \\
		\end{array}
		\right)$ and ${\rm det}\left(\begin{array}{ccc}
		a & x & 0 \\
		x & b & 0 \\
		0 & 0 & 1 \\
		\end{array}
		\right)^{-1}=\dfrac{1}{a b-x^2}.$
		Using these calculation in \eqref{formaC}, we obtain
	\begin{equation}
	(\det U_0(0))\, U_0^{-1}(0)\, \,\begin{footnotesize}\begin{pmatrix}
	0&1&0 \\
	-1&0&0 \\
	0&0&0
	\end{pmatrix}\end{footnotesize}\,  U_0^{-1}(0) \,=\,\begin{footnotesize}\begin{pmatrix}
	0&1&0 \\
	-1&0&0 \\
	0&0&0
	\end{pmatrix}\end{footnotesize}.
	\end{equation}
	Hence, the alternator tensor has the
	representation given in iv).
\end{proof}
\end{footnotesize}

\end{document}